\documentclass[a4paper,11pt]{article}
\usepackage[utf8]{inputenc}
\usepackage[T1]{fontenc}
\usepackage{graphicx}
\usepackage{amsthm,amsmath,amsfonts,amssymb}
\usepackage[numbers]{natbib}
\usepackage{hyperref}
\hypersetup{
  colorlinks,
  linkcolor=blue, %
  citecolor=blue, %
  linktoc=all %
}
\usepackage{graphicx}

\usepackage{mathtools}
\usepackage{bm}
\usepackage{color}
\usepackage{fullpage}

\definecolor{crimson}{rgb}{0.86, 0.08, 0.24}
\definecolor{awesome}{rgb}{1.0, 0.13, 0.32}
\definecolor{newgreen}{rgb}{0.0,0.6,0.0}

\usepackage{relsize}

\renewcommand{\leq}{\leqslant}
\renewcommand{\geq}{\geqslant}
\newcommand{\mleq}{\preccurlyeq}
\newcommand{\mgeq}{\succcurlyeq}

\newcommand{\di}{\mathrm{d}}

\newcommand{\eps}{\varepsilon}

\newcommand{\argmin}{\mathop{\mathrm{arg}\,\mathrm{min}}}

\newcommand{\wt}{\widetilde}
\newcommand{\wh}{\widehat}

\newcommand{\R}{\mathbf R}

\usepackage{xspace}
\newcommand{\ie}{{i.e.}\@\xspace} 
\newcommand{\eg}{e.g.\@\xspace}
\newcommand{\iid}{i.i.d.\@\xspace}

\newcommand{\vspan}{\mathrm{span}}

\renewcommand{\ker}{\mathop{\mathrm{ker}}}

\newcommand{\Id}{I_d}%
\newcommand{\id}{\Id}

\newcommand{\tr}{\mathrm{Tr}}

\newcommand{\opnorm}[1]{\| {#1} \|_{\mathrm{op}}}
\newcommand{\dist}{\mathrm{dist}}

\newcommand{\E}{\mathbb E}
\renewcommand{\P}{\mathbb P}
\newcommand{\Var}{\mathrm{Var}}
\newcommand{\var}{\Var}

\newcommand{\cond}{|}%

\newcommand{\probas}{\mathcal{P}}
\newcommand{\indic}[1]{\bm 1 ( #1 )}
\newcommand{\kl}{\mathrm{KL}}%
\newcommand{\kll}[2]{\kl ({#1}, {#2})}%

\newcommand{\gaussdist}{\mathcal{N}}

\newcommand{\F}{\mathcal{F}}

\newcommand{\loss}{\mathcal{L}}
\newcommand{\risk}{\mathcal{R}}

\newcommand{\excessrisk}{\mathcal{E}}

\newcommand{\ls}[1]{\wh {#1}_n^{\mathrm{LS}}} %
\newcommand{\lsb}{\ls{\beta}}
\newcommand{\ridgeb}{\wh \beta_{\lambda, n}}

\newcommand{\wellclass}{\probas_{\mathrm{well}}}%
\newcommand{\gaussclass}{\probas_{\mathrm{Gauss}}}%

\newcommand{\misclass}{\probas_{\mathrm{mis}}}%

\newcommand{\lamin}{\lambda_{\mathrm{min}}}
\newcommand{\lamax}{\lambda_{\mathrm{max}}}

\theoremstyle{plain}
\newtheorem{proposition}{Proposition}%
\newtheorem{theorem}{Theorem}
\newtheorem{lemma}{Lemma}
\newtheorem{corollary}{Corollary}
\newtheorem{fact}{Fact}

\theoremstyle{definition}
\newtheorem{definition}{Definition}

\newtheorem{assumption}{Assumption}%

\theoremstyle{remark}
\newtheorem{remark}{Remark}

\title{Exact minimax risk for linear least squares, and the lower tail of sample covariance matrices}

\author{Jaouad Mourtada\footnote{CREST, ENSAE, Institut Polytechnique de Paris, France; \texttt{jaouad.mourtada@ensae.fr}}}

\date{}

\begin{document}

\maketitle

  \begin{abstract}
  We consider random-design linear prediction and related questions on the lower tail of random matrices.
  It is known that, under boundedness constraints, the minimax risk is of order $d/n$ in dimension $d$ with $n$ samples.
  Here, we study the minimax expected excess risk over the full linear class, depending on the distribution of covariates.
  First, the least squares estimator is exactly minimax optimal in the well-specified case, for every distribution of covariates.
  We express the minimax risk in terms of the distribution of statistical leverage scores of individual samples,
  and deduce a minimax lower bound of $d/(n-d+1)$ for any covariate distribution, nearly matching the risk for Gaussian design.
  We then obtain sharp nonasymptotic upper bounds for covariates that satisfy a ``small ball''-type regularity condition in both well-specified and misspecified cases.
  
  Our main technical contribution is the study of the lower tail of the smallest singular value of empirical covariance matrices at small values.
  We establish a lower bound on this lower tail, valid for any distribution in dimension $d \geq 2$, together with a matching upper bound under a necessary regularity condition.
  Our proof relies on the PAC-Bayes technique for controlling empirical processes, and extends an analysis of Oliveira devoted to a different part of the lower tail.

\end{abstract}

\section{Introduction}
\label{sec:introduction}

Linear least-squares regression, also called {random-design linear regression} or {linear aggregation}, is one of the basic statistical prediction problems.
Specifically, given a random pair $(X, Y)$ where $X$ is a covariate vector in $\R^d$ and $Y$ is a scalar response, the aim is to predict $Y$ using a linear function $\langle \beta, X\rangle = \beta^\top X$ (with $\beta \in \R^d$) of $X$ as well as possible, in a sense measured by the prediction risk with squared error $R (\beta) = \E [ (Y - \langle \beta, X\rangle)^2 ]$.
The best prediction is achieved by the population risk minimizer $\beta^*$, which equals:
\begin{equation*}
  \beta^* = \Sigma^{-1} \E [Y X]
\end{equation*}
where $\Sigma := \E [X X^\top]$, assuming that both $\Sigma$ and $\E [YX]$ are well-defined and that $\Sigma$ is invertible.
In the statistical setting considered here, the joint distribution $P$ of the pair $(X,Y)$ is unknown.
The goal is then, given a sample $(X_1, Y_1), \dots, (X_n, Y_n)$ of $n$ \iid realizations of $P$, to find a predictor (also called \emph{estimator}) $\wh \beta_n$ with small \emph{excess risk}
\begin{equation*}
  \excessrisk_P (\wh \beta_n)
  := R (\wh \beta_n) - R (\beta^*)
  = \| \wh \beta_n - \beta^* \|_\Sigma^2
  \, ,
\end{equation*}
where we define $\| \beta \|_\Sigma^2 := \langle \Sigma \beta, \beta \rangle = \| \Sigma^{1/2} \beta \|^2$.
Arguably the most common procedure is the \emph{Ordinary Least Squares} (OLS) estimator (that is, the empirical risk minimizer), defined by
\begin{equation*}
  \lsb := \argmin_{\beta \in \R^d} \bigg\{ \frac{1}{n} \sum_{i=1}^n ( Y_i - \langle \beta, X_i \rangle )^2 \bigg\}
  = \wh \Sigma_n^{-1} \cdot \frac{1}{n} \sum_{i=1}^n Y_i X_i
  \, ,
\end{equation*}
with $\wh \Sigma_n := n^{-1} \sum_{i=1}^n X_i X_i^\top$ the sample covariance matrix.

Linear classes are of particular importance to regression problems, both in themselves and since they naturally appear in the context of nonparametric estimation \cite{gyorfi2002nonparametric,tsybakov2009nonparametric}.
In this note, we analyze this problem from a decision-theoretic perspective, focusing on the 
minimax excess risk with respect to the full linear class $\F = \{ x \mapsto \langle \beta, x\rangle : \beta \in \R^d \}$, and in particular on its dependence on the distribution of $X$.
The minimax perspective is relevant when little is known or assumed on the optimal parameter $\beta^*$.
Specifically, define the \emph{minimax excess risk} (see, e.g., \cite{lehmann1998tpe}) with respect to $\F$ under a set $\probas$ of joint distributions $P$ on $(X,Y)$ as:
\begin{equation}
  \label{eq:minimax-excess-risk-linear}
  \inf_{\wh \beta_n} \sup_{P \in \probas} \E [ \excessrisk_P (\wh \beta_n) ]
  = \inf_{\wh \beta_n} \sup_{P \in \probas} \Big( \E [ R (\wh \beta_n) ] - \inf_{\beta \in \R^d} R (\beta) \Big)
  \, ,
\end{equation}
where the infimum in~\eqref{eq:minimax-excess-risk-linear} spans over all estimators $\wh \beta_n$ based on $n$ samples, while the expectation and the risk $R$ depend the underlying distribution $P$.
Our aim is to characterize the influence of the distribution $P_X$ of covariates on the hardness of the problem.
Hence, our considered classes $\probas$ of distributions are obtained by fixing the marginal distribution of $X$, and letting the optimal regression parameter $\beta^*$ vary freely in $\R^d$ (see Section~\ref{sec:minimax-ls}).

Some minimal regularity condition on the distribution $P_X$ is required to ensure even finiteness of the minimax risk~\eqref{eq:minimax-excess-risk-linear} in the random-design setting.
Indeed, assume that the distribution $P_X$ charges some positive mass on a hyperplane $H \subset \R^d$ (we call such a distribution \emph{degenerate}, see Definition~\ref{def:degenerate}).
Then, with positive probability, all points $X_1, \dots, X_n$ in the sample lie within $H$, so that the component of the optimal parameter $\beta^*$ which is orthogonal to $H$ cannot be estimated.
However, this component matters for out-of-sample prediction, in case the point $X$ for which one wishes to compute prediction does not belong to $H$.
Such a degeneracy (or quantitative variants, where $P_X$ puts too much mass at the neighborhood of a hyperplane) turns out to be the main obstruction to achieving controlled uniform excess risk over~$\R^d$.

The second part of this note (Section~\ref{sec:small-ball-bounds}) is devoted to the study of the \emph{sample covariance matrix}
\begin{equation}
  \label{eq:def-sample-covariance}
  \wh \Sigma_n := \frac{1}{n} \sum_{i=1}^n X_i X_i^\top
  \, ,
\end{equation}
where $X_1, \dots, X_n$ are \iid samples from $P_X$.
Indeed, upper bounds on the minimax risk require a control of relative deviations of the empirical covariance matrix $\wh \Sigma_n$ with respect to its population counterpart $\Sigma$, in the form of \emph{negative moments} of the rescaled covariance matrix $\wt \Sigma_n := \Sigma^{-1/2} \wh \Sigma_n \Sigma^{-1/2}$, namely
\begin{equation}
  \label{eq:neg-moments}
  \E [ \lamin(\wt \Sigma_n)^{-q} ]
\end{equation}
where $q \geq 1$ and $\lamin(A)$ is the smallest eigenvalue of symmetric matrix $A$.

Control of lower relative deviations of $\wh \Sigma_n$ with respect to $\Sigma$ can be expressed in terms of lower-tail bounds, of the form
\begin{equation}
  \label{eq:tail-bound-lamin}
  \P \big( \lamin ( \wt \Sigma_n ) \leq t \big)
  \leq \delta
  \, ,
\end{equation}
where $t, \delta \in (0, 1)$.
Sub-Gaussian tail bounds for $\lamin( \wt \Sigma_n)$, of the form~\eqref{eq:tail-bound-lamin} with
\begin{equation*}
  \delta = \exp \Big({- c n \Big(1- C \sqrt{\frac{d}{n}} - t \Big)_+^2} \Big)
\end{equation*}
for some constants $c, C$ depending on $P_X$, as well as similar bounds for the largest eigenvalue $\lamax ( \wt \Sigma_n)$, can be obtained under the (strong) assumption that $X$ is sub-Gaussian (see, \eg, \cite{vershynin2012introduction}).
Remarkably, it is shown in \cite{oliveira2016covariance,koltchinskii2015smallest} that such bounds can be obtained for the \emph{smallest} eigenvalue under much weaker assumptions on $X$, namely bounded fourth moments of linear marginals of $X$.

While sub-Gaussian bounds provide a precise control of deviations~\eqref{eq:tail-bound-lamin} for $t \in (c, 1 - C \sqrt{d/n})$ (for some constants $c, C$), they do not suffice to control moments of $\lamin (\wt \Sigma_n)^{-1}$.
Indeed, such bounds ``saturate'' in the sense that $\delta = \delta (t)$ does not tend to $0$ as $t \to 0$; in other words, they provide no nonvacuous guarantee~\eqref{eq:tail-bound-lamin} with $t > 0$ as the confidence level $1-\delta$ tends to~$1$.
This prevents one from integrating such tail bounds and deduce a control of moments of the form~\eqref{eq:neg-moments}.
In fact, the covariance matrix of a sub-Gaussian matrix can be singular with positive probability (exponentially small in $n$), for instance for matrices with independent Bernoulli entries; in order to ensure invertibility at all confidence levels, different regularity assumptions are required.
In Section~\ref{sec:small-ball-bounds}, we complement the sub-Gaussian tail bounds by a study of non-asymptotic large deviation bounds~\eqref{eq:tail-bound-lamin} with $\delta = \exp ( - n \psi (t) )$ for small values of $t$, namely $t \in (0, c)$.

\subsection{Summary of contributions}
\label{sec:summ-contrib}

Below is an overview of our results on least squares regression, which appear in Section~\ref{sec:minimax-ls}:

\begin{enumerate}
\item We determine the minimax excess risk in the {well-specified case} (where the true regression function $x \mapsto \E [ Y \cond X=x]$ is linear) for every distribution $P_X$ of features and noise level $\sigma^2$.
  For some ``degenerate'' distributions (Definition~\ref{def:degenerate}), the minimax risk is infinite (Proposition~\ref{prop:degenerate-minimax-risk}); while for non-degenerate ones, the OLS estimator is exactly minimax (Theorem~\ref{thm:ls-minimax}) irrespective of $P_X, \sigma^2$.
\item We express the minimax risk in terms of the distribution of \emph{statistical leverage scores} of samples drawn from $P_X$ (Theorem~\ref{thm:minimax-leverage-score}).
  Quite intuitively, distributions of $X$ for which leverage scores are uneven are seen to be harder from a minimax point of view.
  We deduce a precise minimax lower bound of $\sigma^2 d / (n - d + 1)$, valid for every distribution $P_X$ of covariates.
  This lower bound nearly matches the $\sigma^2 d / (n - d - 1)$ risk for centered Gaussian covariates, in both low ($d/n \to 0$) and moderate ($d/n \to \gamma \in (0, 1)$) dimensions; hence, Gaussian covariates are almost the ``easiest'' ones in terms of minimax risk.
  This provides a counterpart to results obtained in the moderate-dimensional regime for \emph{independent} covariates from the Marchenko-Pastur law.
\item We then turn to {upper bounds} on the minimax risk.
  Under some quantitative variant of the non-degeneracy assumption (Assumption~\ref{ass:small-ball}) together with a fourth-moment condition on $P_X$ (Assumption~\ref{ass:fourth-moment-norm} or~\ref{ass:equiv-l4-l2}), we show that the minimax risk is finite and scales as $(1 + o(1)) \sigma^2 d/n$ when $d = o (n)$, both in the well-specified (Theorem~\ref{thm:upper-well-spec}) and misspecified (Proposition~\ref{prop:upper-misspecified}) cases.
  In particular, OLS is asymptotically minimax in the misspecified case as well, as $d/n \to 0$.
  To our knowledge, this gives the first bounds on the expected risk of the OLS estimator for general random design distribution.
\end{enumerate}
The previous upper bounds rely on the study of the lower tail of the sample covariance matrix $\wh \Sigma_n$, carried out in Section~\ref{sec:small-ball-bounds}.
Our contributions here are the following (assuming, to simplify notation, that $\E [X X^\top] = I_d$):
\begin{enumerate}
\item[4.] First, we establish a \emph{lower bound} on the lower tail of $\lamin (\wh \Sigma_n)$, for $d \geq 2$ and \emph{any} distribution $P_X$ such that $\E [ X X^\top ] = I_d$, of the form: $\P (\lamin (\wh \Sigma_n) \leq t) \geq (c t)^{n/2}$ for some numerical constant $c$ and every $t \in (0, 1)$ (Proposition~\ref{prop:small-ball-lower}).
  We also exhibit a ``small-ball'' condition (Assumption~\ref{ass:small-ball}) which is necessary to achieve similar upper bounds.
\item[5.] Under Assumption~\ref{ass:small-ball}, we show a matching \emph{upper bound} on the lower tail $\P (\lamin (\wh \Sigma_n) \leq t)$, valid {for all $t\in (0, 1)$}, and in particular for small $t$.
  This result (Theorem~\ref{thm:lower-tail-covariance}) is the core technical contribution of this paper.
  Its proof relies the PAC-Bayesian technique for controlling empirical processes, which was used by \cite{oliveira2016covariance} to control a different part of the lower tail;
  however, some non-trivial refinements (such as non-Gaussian smoothing) are needed to handle small values of $t$.
  This result can be equivalently stated as an upper bound on moments of $\lamin (\wh \Sigma_n)^{-1}$, namely $\| \lamin (\wh \Sigma_n)^{-1} \|_{L^q} = O (1)$ for $q \asymp n$ (Corollary~\ref{cor:negative-moments-smallest}).
\item[6.] Finally, we discuss in Section~\ref{sec:small-ball-indep} the case of independent covariates.
  In this case, the ``small-ball'' condition (Assumption~\ref{ass:small-ball}) holds naturally under mild regularity assumptions on the distribution of individual coordinates.
  A result of \cite{rudelson2014small} establishes this for coordinates with bounded density;
  we complement it by a general anti-concentration result for linear combination of independent variables (Proposition~\ref{prop:uniform-small-ball}),
  implying Assumption~\ref{ass:small-ball} for sufficiently ``non-atomic'' coordinates.
\end{enumerate}

\subsection{Related work}
\label{sec:related-work}

Linear least squares regression is a classical problem, and the literature on this topic is too vast to be surveyed here; we
refer to \cite{gyorfi2002nonparametric,audibert2010linear,hsu2014ridge} (and references therein) for a more thorough overview.
In addition, while we focus on mean-squared prediction error, different criteria can be considered, as in the predictive inference literature \cite{rinaldo2019bootstrapping}.
Analysis of least squares regression is most standard and straightforward in the \emph{fixed design} setting, where the covariates $X_1, \dots, X_n$ are
deterministic and the risk is evaluated within-sample;
in this case, the expected excess risk of the OLS estimator is bounded by $\sigma^2 d / n$ (see, \eg, \cite{hsu2014ridge}).

In the random design setting considered here, a classical result \cite[Theorem~11.3]{gyorfi2002nonparametric}
states that, if $\var (Y \cond X) \leq \sigma^2$ and the true regression function $g^* (x) = \E [Y \cond X = x] $ satisfies $| g^* (X) | \leq L$ almost surely, then the risk $R (g) = \E [ (g(X) - Y)^2 ]$ of the (nonlinear) \emph{truncated} ERM estimator, defined by $\wh g^L_n (x) = \min (-L, \max (L, \langle \lsb, x\rangle))$, is at most
\begin{equation}
  \label{eq:bound-gyorfi-truncated}
  \E [ R (\wh g_n^{L}) ] - R (g^*)
  \leq 8 \big( R (\beta^*) - R (g^*) \big) + C \max (\sigma^2, L^2) \frac{d (\log n + 1)}{n}
\end{equation}
for some universal constant $C > 0$.
This result is an \emph{inexact oracle inequality}, where the risk
is bounded by a constant times that of the best linear predictor $\beta^*$.
Such guarantees are adequate %
in a nonparametric setting,
where the approximation error $R (\beta^*) - R (g^*)$ of the linear model is itself of order $O(d/n)$ \cite{gyorfi2002nonparametric}.
On the other hand, when no assumption is made on the magnitude of the approximation error, this bound does not ensure that the risk of the estimator approaches that of $\beta^*$.
By contrast, in the \emph{linear aggregation} problem as defined by \cite{nemirovski2000nonparametric} (and studied by \cite{tsybakov2003optimal,catoni2004statistical,bunea2007aggregation,audibert2011robust,hsu2014ridge,lecue2016performance,mendelson2015smallball,oliveira2016covariance}), one seeks to obtain excess risk bounds, also called \emph{exact} oracle inequalities (where the constant $8$ in the bound~\eqref{eq:bound-gyorfi-truncated} is replaced by $1$), with respect to the linear class.
In this setting, Tsybakov \cite{tsybakov2003optimal} showed that the minimax rate of aggregation is of order $O (d/n)$, under boundedness assumptions on the regression function and on covariates.
It is also worth noting that bounds on the regression function also implicitly constrain the optimal regression parameter to lie in some ball.
This contrasts with the approach considered here, where minimax risk with respect to the full linear class is considered.
Perhaps most different from the point of view adopted here is the approach from \cite{foster1991prediction,vovk2001competitive,azoury2001relative,shamir2015sample,bartlett2015minimax},
whose authors consider worst-case covariates (either in the individual sequences  or in the agnostic learning setting) under boundedness assumptions on both covariates and outputs, and investigate achievable excess risk (or regret) bounds with respect to bounded balls in this case.
By contrast, we take the distribution of covariates as given and allow the optimal regression parameter to be arbitrary, and study under which conditions on the covariates uniform bounds are achievable.
Another type of non-uniform guarantees over linear classes is achieved by Ridge regression \cite{hoerl1962ridge,tikhonov1963regularization} in the context of reproducing kernel Hilbert spaces  \cite{cucker2002best,cucker2002mathematical,devito2005model,caponnetto2007optimal,smale2007integral,steinwart2009optimal,audibert2011robust,hsu2014ridge}, where the bounds do not depend explicitly on the dimension $d$, but rather on spectral properties of $\Sigma$ and some norm of $\beta^*$.

This work is concerned with the expected risk.
Risk bounds in probability are obtained, among others, by \cite{audibert2011robust,hsu2014ridge,hsu2016heavy,oliveira2016covariance,mendelson2015smallball,lecue2016performance}.
While such bounds hold with high probability, the probability %
is upper bounded and cannot be arbitrarily close to $1$, so that they cannot be integrated to control the expected risk.
Indeed, some additional regularity conditions are required in order to have finite minimax risk, as will be seen below.
To the best of our knowledge, the only available uniform expected risk bounds for random-design regression are obtained in the case of Gaussian covariates, where they rely on the knowledge of the closed-form distribution of inverse covariance matrices \cite{stein1960multiple,breiman1983many,anderson2003introduction}.
One reason for considering the expected risk is that it is a single scalar, which can be more tightly controlled (in terms of matching upper and lower bounds) and compared across distributions than quantiles.
In addition, random-design linear regression is a classical statistical problem, which justifies its precise decision-theoretic analysis.
On the other hand, expected risk only provides limited information on the tails of the risk in the high-confidence regime: in the case of heavy-tailed noise, the OLS estimator may perform poorly, and dedicated robust estimators may be required (see, \eg, \cite{audibert2011robust} and the references in \cite{lugosi2019survey}).

Another line of work \cite{elkaroui2013asymptotic,dicker2016ridge,donoho2016robust-amp,elkaroui2018predictor,dobriban2018ridge} considers the limiting behavior of regression procedures in the high-dimensional asymptotic regime where $d, n$ tend to infinity at a proportional rate, with their ratio kept constant \cite{huber1973robust}.
The results in this setting take the form of a convergence in probability of the risk to a limit depending on the ratio $d/n$ as well as the properties of $\beta^*$.
With the notable exception of \cite{elkaroui2018predictor}, the previous results hold under the assumption that the covariates are either Gaussian, or have a joint independence structure that leads to the same limiting behavior in high dimension.
In contrast, here we consider non-asymptotic bounds valid for fixed $n,d$, general design distribution and uniformly over $\beta^* \in \R^d$.

The study of spectral properties of sample covariance matrices has a rich history (see for instance \cite{bai2010spectral,anderson2010introduction,tao2012topics} and references therein); we refer to \cite{rudelson2010nonasymptotic} for an overview of results (up to 2010) on the non-asymptotic control of the smallest eigenvalue of sample covariance matrices, which is the topic of Section~\ref{sec:small-ball-bounds}.
It is well-known \cite{vershynin2012introduction} that sub-Gaussian tail bounds on both the smallest and largest eigenvalues can be obtained under sub-Gaussian assumptions on covariates (see also \cite{koltchinskii2017operators} for operator norm concentration under general population covariance).
A series of work obtained control on these quantities under weaker assumptions \cite{adamczak2010logconcave,mendelson2014singular,srivastava2013covariance,tikhomirov2018sample}.
A key observation, which has been exploited in a series of work \cite{srivastava2013covariance,koltchinskii2015smallest,oliveira2016covariance,yaskov2014lower,yaskov2015sharp,vandegeer2014higher}, is that the smallest eigenvalue can be controlled under much weaker tail assumptions than the largest one.
Our study follows this line of work, but considers a different part of the lower tail, which poses additional technical difficulties;
we also provide a general lower bound on the lower tail.

\paragraph{Notation.}

Throughout this text, the transpose of an $m \times n$ real matrix $A$ is denoted $A^\top$, its trace (when $m= n$) $\tr (A)$, and vectors in $\R^d$ are identified with $d \times 1$ column vectors.
In addition, the coordinates of a vector $x \in \R^d$ are indicated as superscripts: $x = (x^j)_{1 \leq j \leq d}$.
We also denote $\langle x, z\rangle = x^\top z = \sum_{j=1}^d (x^j) \cdot (z^j)$ the canonical scalar product of $x, z \in \R^d$, and $\| x \| = \langle x, x\rangle^{1/2}$ the associated Euclidean norm.
In addition, for any symmetric and positive $d \times d$ matrix $A$, we define the scalar product $\langle x, z\rangle_A = \langle A x, z\rangle$ and norm $\| x \|_A = \langle A x, x\rangle^{1/2} = \| A^{1/2} x \|$.
The $d \times d$ identity matrix is denoted $I_d$, while $S^{d-1} = \{ x \in \R^d : \| x \| = 1 \}$ refers to the unit sphere.
The smallest and largest eigenvalues of a symmetric matrix $A$ are denoted $\lamin (A)$ and $\lamax (A)$ respectively; if $A$ is positive definite, then $\lamax (A) = \opnorm{A}$
is the operator norm of $A$ (with respect to $\|\! \cdot\! \|$), while $\lamin (A) = \opnorm{A^{-1}}^{-1}$.
We denote by $\dist (x, A) = \inf_{y \in A} \| x - y \|$ the distance of $x \in \R^d$ to a subset $A \subset \R^d$.

\section{Exact minimax analysis of least-squares regression}
\label{sec:minimax-ls}

This section is devoted to the minimax analysis of the linear least-squares problem, and in particular on the dependence of its hardness on the distribution $P_X$ of covariates.
In Section~\ref{sec:minim-risk-analys}, we indicate the exact minimax risk and estimator in the {well-specified} case, namely on the class $\wellclass (P_X, \sigma^2)$.
In Section~\ref{sec:leverage-score-lower}, we express the minimax risk in terms of the distribution of statistical leverage scores, and deduce a general lower bound.
Finally, Section~\ref{sec:upper-bound-minimax} provides {upper bounds} on the minimax risk under some regularity condition on the distribution $P_X$, both in the well-specified and misspecified cases.

Throughout this note, we assume that the covariate vector $X$ satisfies $\E [ \| X \|^2 ] < + \infty$, and denote $\Sigma = \E [X X^\top]$ its covariance matrix (by a slight but common abuse of terminology, we refer to $\Sigma$ as the covariance matrix of $X$ even when $X$ is not centered).
In addition, we assume that $\Sigma$ is invertible, or equivalently that the support of $X$ is not contained in any hyperplane;
this assumption is not restrictive (up to restricting to the span of the support of $X$, a linear subspace of $\R^d$) and only serves to simplify notations.
Then, for every distribution of $Y$ given $X$ such that $\E [Y^2] < + \infty$, the risk $R (\beta) = \E [ (\langle \beta, X\rangle - Y)^2 ]$ of any $\beta \in \R^d$ is finite; this risk is uniquely minimized by $\beta^* = \Sigma^{-1} \E [Y X]$, where $\E [Y X]$ is well-defined since $\E [ \| Y X \| ] \leq \E [Y^2]^{1/2} \E [\| X \|^2]^{1/2} < + \infty$.
The response $Y$ then writes
\begin{equation}
  \label{eq:decomposition-linear}
  Y = \langle \beta^*, X\rangle + \eps
  \, ,
\end{equation}
where $\eps$ is the \emph{error}, with
$\E [ \eps X] = \E [Y X] - \Sigma \beta^* = 0$.
The distribution $P$ of $(X, Y)$ is then characterized by the distribution $P_X$ of $X$, the coefficient $\beta^* \in \R^d$ as well as the conditional distribution of $\eps$ given $X$, which satisfies $\E [\eps^2] \leq \E [Y^2] < + \infty$ and $\E [ \eps X ] = 0$.
Now, given a distribution $P_X$ of covariates and a bound $\sigma^2$ on the conditional second moment of the error, define the following three classes,
where $Y$ is given by~\eqref{eq:decomposition-linear}:
\begin{align}
  \gaussclass (P_X, \sigma^2)
  &= \Big\{ P_{(X, Y)} : X \sim P_X, \, \beta^* \in \R^d, \, \eps \cond X \sim \gaussdist (0, \sigma^2) \Big\} \nonumber \\
  \wellclass (P_X, \sigma^2)
  &= \Big\{ P_{(X, Y)} : X \sim P_X, \, \beta^* \in \R^d, \,  \E [ \eps \cond X ] = 0, \, \E [ \eps^2 \cond X] \leq \sigma^2 \Big\} \nonumber \\
  \misclass (P_X, \sigma^2)
  &= \Big\{ P_{(X, Y)} : X \sim P_X, \, \beta^* \in \R^d, \, \E [ \eps^2 \cond X] \leq \sigma^2 \Big\} \, . 
\end{align}
The class $\gaussclass$ corresponds to the standard case of independent Gaussian noise, while $\wellclass$ includes all \emph{well-specified} distributions, such that the true regression function $x \mapsto \E [Y \cond X = x]$ is linear.
Finally, $\misclass$ corresponds to the general \emph{misspecified} case, where  the regression function $x \mapsto \E [Y \cond X = x]$ is not assumed to be linear.

\subsection{Minimax analysis of linear least squares}
\label{sec:minim-risk-analys}

We start with the following definition.
\begin{definition}
  \label{def:degenerate}
  The distribution $P_X$ on $\R^d$ is \emph{degenerate} if there exists a linear hyperplane $H \subset \R^d$ such that $\P (X \in H) > 0$ (that is, if there exists some $\theta \in S^{d-1}$ such that $\P (\langle \theta, X\rangle = 0) > 0$).
\end{definition}

\begin{fact}
  \label{prop:equiv-degenerate}
  Let $n \geq d$.
  The following properties are equivalent:
  \begin{enumerate}
  \item The distribution $P_X$ is non-degenerate;
  \item The sample covariance matrix $\wh \Sigma_n$ is invertible almost surely;
  \item The ordinary least-squares (OLS) estimator
    \begin{equation}
      \label{eq:def-ols}
      \lsb := \argmin_{\beta \in \R^d} \sum_{i=1}^n (\langle \beta, X_i\rangle - Y_i)^2
    \end{equation}
    is uniquely defined almost surely, and equals $\lsb = \wh \Sigma_n^{-1} n^{-1} \sum_{i=1}^n Y_i X_i$.
  \end{enumerate}
\end{fact}

\begin{proof}
  The equivalence between the second and third points is standard: the empirical risk being convex, its global minimizers are the critical points $\beta$ characterized by $\wh \Sigma_n \beta = n^{-1} \sum_{i=1}^n Y_i X_i$.

  We now prove that the second point implies the first, by contraposition.
  If $\P (\langle \theta, X\rangle = 0) = p > 0$ for some $\theta \in S^{d-1}$, then with probability $p^n$, $\langle \theta, X_i\rangle = 0$ for $i=1, \dots, n$, so that $\wh \Sigma_n \theta = n^{-1} \sum_{i=1}^n \langle \theta, X_i \rangle X_i = 0$ and thus $\wh \Sigma_n$ is not invertible.

  Conversely, let us show that the first point implies the second one.
  Note that the latter amounts to saying that $X_1, \dots, X_n$ span $\R^d$ almost surely.
  It suffices to show this for $n = d$, which we do by showing that, almost surely, $V_k = \vspan (X_1, \dots, X_k)$ is of dimension $k$ for $0 \leq k \leq d$, by induction on $k$.
  The case $k=0$ is clear.
  Now, assume that $k \leq d$ and that $V_{k-1}$ is of dimension $k-1 \leq d-1$ almost surely.
  Then, $V_{k-1}$ is contained in a hyperplane of $\R^d$, and since $X_k$ is independent of $V_{k-1}$, the first point implies that $\P (X_k \in V_{k-1} ) = 0$, so that $V_k$ is of dimension $k$ almost surely.
\end{proof}

\begin{remark}[Intercept]
  Assume that $X = (X^j)_{1\leq j \leq d}$, where $X^d \equiv 1$ is an intercept variable.
  Then, the distribution $P_X$ is degenerate if and only if there exists $\theta = (\theta^j)_{1\leq j < d} \in \R^{d-1} \setminus \{ 0 \}$ and $c \in \R$ such that $\sum_{j=1}^{d-1} \theta^j X^j = c$ with positive probability.
  This amounts to say that $(X^1, \dots, X^{d-1})$ belongs to some fixed affine hyperplane of $\R^{d-1}$ with positive probability.
\end{remark}

The following result shows that non-degeneracy of the design distribution is necessary to obtain finite minimax risk.

\begin{proposition}[Degenerate case]
  \label{prop:degenerate-minimax-risk}
  Assume that either $n < d$, or that the distribution $P_X$ of $X$ is degenerate, in the sense of Definition~\ref{def:degenerate}.
  Then, the minimax excess risk with respect to the class $\gaussclass (P_X, \sigma^2)$ is infinite.
\end{proposition}

An infinite minimax excess risk means that some dependence on the true parameter $\beta^*$ (for instance, through its norm) is unavoidable in the expected risk of any estimator $\wh \beta_n$.
From now on and until the rest of this section, we assume that the distribution $P_X$ is non-degenerate and that $n \geq d$.
In particular, the OLS estimator is well-defined, and the empirical covariance matrix $\wh \Sigma_n$ is invertible almost surely.
Theorem~\ref{thm:ls-minimax} below provides the exact minimax excess risk and estimator in the well-specified case.

\begin{theorem}
  \label{thm:ls-minimax}
  Assume that $P_X$ is non-degenerate and $n \geq d$.
  The minimax risks over classes $\wellclass (P_X, \sigma^2)$ and $\gaussclass (P_X, \sigma^2)$ coincide, and equal
  \begin{align}
    \label{eq:minimax-risk-ls}
    \inf_{\wh \beta_n} \sup_{P \in \wellclass (P_X, \sigma^2)} \E
    \big[ \excessrisk_P (\wh \beta_n) \big]
    = \frac{\sigma^2}{n} \cdot \E \big[ \tr (\wt \Sigma_n^{-1}) \big]
  \end{align}
  where $\wt \Sigma_n = \Sigma^{-1/2} \wh \Sigma_n \Sigma^{-1/2}$
  is the rescaled empirical covariance matrix.
  In addition, the minimax risk is achieved by the OLS estimator~\eqref{eq:def-ols} over the classes $\gaussclass (P_X, \sigma^2)$ and $\wellclass (P_X, \sigma^2)$ for every $P_X$ and $\sigma^2$.
\end{theorem}

The proof of Theorem~\ref{thm:ls-minimax} and Proposition~\ref{prop:degenerate-minimax-risk} is provided in Section~\ref{sec:proof-minimax-risk}, and relies on standard decision-theoretic arguments (see \cite[Chapter~2]{tsybakov2009nonparametric} and \cite[Section~4.10]{johnstone2019gaussian}).
First, an upper bound (in the non-degenerate case) over $\wellclass (P_X, \sigma^2)$ is obtained for the OLS estimator.
Then, a matching lower bound on the minimax risk over the subclass $\gaussclass (P_X, \sigma^2)$ is established by considering the Bayes risk under Gaussian prior on $\beta^*$ and using a monotone convergence argument.

\begin{remark}[Linear changes of covariates]
  \label{rem:lin-transf}
  The minimax risk is invariant under invertible linear transformations of the covariates $x$.
  This can be seen a priori, by noting that the class of linear functions of $x$ is invariant under linear changes of variables.
  To recover it from Theorem~\ref{thm:ls-minimax}, let $X' = A X$, where $A$ is an invertible $d \times d$ matrix.
  Since $\Sigma' = \E [ X' X'^\top ]$ equals $A \Sigma A^\top$ and $\wh \Sigma_n' = n^{-1} \sum_{i=1}^n X_i' X_i'^\top$ equals $A \wh \Sigma_n A^\top$, we have
  \begin{equation*}
    \wh \Sigma_n'^{-1} \Sigma'
    = ((A^\top)^{-1} \wh \Sigma_n^{-1} A^{-1}) (A \Sigma A^\top)
    = (A^\top)^{-1} (\wh \Sigma_n^{-1} \Sigma) A^\top
    \, ,
  \end{equation*}
  which is conjugate to $\wh \Sigma_n^{-1} \Sigma$ and hence has the same trace; this concludes by Theorem~\ref{thm:ls-minimax} (as $\tr (\wt \Sigma_n^{-1}) = \tr (\wh \Sigma_n^{-1} \Sigma)$).
  In particular, the minimax risk for the design $X$ is the same as the one for
$\wt X = \Sigma^{-1/2} X$.
\end{remark}

Note that the OLS estimator $\lsb$ is minimax optimal for every distribution of covariates $P_X$ and noise level $\sigma^2$.
This shows in particular that the knowledge of neither of those properties of the distribution is helpful to achieve improved risk uniformly over the linear class.
On the other hand, when additional knowledge on the optimal parameter $\beta^*$ is available, OLS may no longer be optimal, and knowledge of $\sigma^2$ may be helpful.

Another consequence of Theorem~\ref{thm:ls-minimax} is that independent Gaussian noise is the least favorable noise structure (in terms of minimax risk) in the well-specified case for a given noise level $\sigma^2$.

Finally, the convexity of the map $A \mapsto \tr (A^{-1})$ on positive matrices \cite{bhatia2009pdmatrices} implies (by Jensen's inequality combined with the identity $\E [ \wt \Sigma_n ] = \id$)
that the minimax risk~\eqref{eq:minimax-risk-ls} is always at least as large as $\sigma^2 d / n$, which is the minimax risk in the fixed-design case.
We will however show in what follows that a strictly better lower bound can be obtained for $d \geq 2$.

\subsection{Connection with statistical leverage and
  distribution-independent lower bound}
\label{sec:leverage-score-lower}

In this section, we provide another expression for the minimax risk over the classes $\wellclass (P_X, \sigma^2)$ and $\gaussclass (P_X, \sigma^2)$, by relating it to the notion of \emph{statistical leverage score} \cite{hoaglin1978hat,chatterjee1988sensitivity,huber1981robust}.

\begin{theorem}[Minimax risk and leverage score]
  \label{thm:minimax-leverage-score}
  Under the assumptions of Theorem~\ref{thm:ls-minimax}, the minimax risk~\eqref{eq:minimax-risk-ls} over the classes $\wellclass (P_X, \sigma^2)$ and $\gaussclass (P_X, \sigma^2)$ is equal to
  \begin{equation}
    \label{eq:minimax-leverage-score}
    \inf_{\wh \beta_n} \sup_{P \in \gaussclass (P_X, \sigma^2)} \E
    \big[ \excessrisk_P (\wh \beta_n) \big]
    =
    {\sigma^2} \cdot \E \bigg[ \frac{\wh \ell_{n+1}}{1 - \wh \ell_{n+1}} \bigg]
  \end{equation}
  where the expectation holds over an \iid sample $X_1, \dots, X_{n+1}$ drawn from $P_X$, and where $\wh \ell_{n+1}$ denotes the \emph{statistical leverage score} of $X_{n+1}$ among $X_1, \dots, X_{n+1}$, defined by:
  \begin{equation}
    \label{eq:def-least-squares}
    \wh \ell_{n+1}
    = \bigg\langle \bigg( \sum_{i=1}^{n+1} X_i X_i^\top \bigg)^{-1} X_{n+1}, X_{n+1} \bigg\rangle
    \, .
  \end{equation}
\end{theorem}

The leverage score $\wh \ell_{n+1}$ of $X_{n+1}$ among $X_1, \dots, X_{n+1}$ measures the influence of the response $Y_{n+1}$ on the associated fitted value  $\wh Y_{n+1} = \langle \wh \beta_{n+1}^{\mathrm{LS}}, X_{n+1}\rangle$: $\wh Y_{n+1}$ is an affine function of $Y_{n+1}$, with slope $\wh \ell_{n+1} = \partial \wh Y_{n+1} / \partial Y_{n+1}$ \cite{hoaglin1978hat,chatterjee1988sensitivity}.
Theorem~\ref{thm:minimax-leverage-score} shows that the minimax predictive
risk under the distribution $P_X$ is characterized by the distribution of leverage scores of samples drawn from this distribution.
Intuitively, uneven leverage scores (with some points having higher leverage) imply that the estimator $\lsb$ is determined by a smaller number of points, and therefore has higher variance.
This is consistent with the message from robust statistics that points with high leverage (typically seen as outliers) can be detrimental to the performance of the least squares estimator \cite{hoaglin1978hat,chatterjee1988sensitivity,huber1981robust}, see also \cite{raskutti2016statistical}.

\begin{proof}[Proof of Theorem~\ref{thm:minimax-leverage-score}]
  By Theorem~\ref{thm:ls-minimax}, the minimax risk over $\gaussclass (P_X, \sigma^2)$ and $\wellclass (P_X, \sigma^2)$ equals, letting $X_{n+1} \sim P_X$ be independent from $X_1, \dots, X_n$:
  \begin{align}    
    \frac{\sigma^2}{n} \cdot \E \big[ \tr (\wt \Sigma_n^{-1}) \big]
    &= \frac{\sigma^2}{n} \cdot \E \big[ \tr (\wh \Sigma_n^{-1} \Sigma) \big] \nonumber \\
    &= \sigma^2 \cdot \E \big[ \tr \big( (n \wh \Sigma_n)^{-1} X_{n+1} X_{n+1}^\top \big) \big] \nonumber \\
    &= \sigma^2 \cdot \E \big[ \langle (n \wh \Sigma_n)^{-1} X_{n+1}, X_{n+1} \rangle \big] \nonumber \\
    &= \sigma^2 \cdot \E \bigg[ \frac{\langle (n \wh \Sigma_n + X_{n+1} X_{n+1}^\top)^{-1} X_{n+1}, X_{n+1} \rangle}{1 - \langle (n \wh \Sigma_n + X_{n+1} X_{n+1}^\top)^{-1} X_{n+1}, X_{n+1} \rangle} \bigg]
      \label{eq:proof-leverage-ls-sm} \\
    &= \sigma^2 \cdot \E \bigg[ \frac{\wh \ell_{n+1}}{1 - \wh \ell_{n+1}} \bigg]
      \nonumber
      \, ,
  \end{align}
  where~\eqref{eq:proof-leverage-ls-sm} follows from
  Lemma~\ref{lem:sherman-morrison} below, with $S = n \wh \Sigma_n$ and $v = X_{n+1}$.
\end{proof}

\begin{lemma}
  \label{lem:sherman-morrison}
  For any symmetric positive $d \times d$ matrix $S$ and $v \in \R^d$,
  \begin{equation}
    \label{eq:sherman-morrison}
    \langle S^{-1} v, v \rangle
    = \frac{\langle (S + vv^\top)^{-1} v, v\rangle}{1 - \langle (S + vv^\top)^{-1} v, v\rangle}
    \, .
  \end{equation}
\end{lemma}

\begin{proof}
  Since $S + v v^\top \mgeq S$ is positive, it is invertible, and the Sherman-Morrison formula \cite{horn1990matrix} shows that
  \begin{align*}
    (S + v v^\top)^{-1}
    &= S^{-1} - \frac{S^{-1} v v^\top S^{-1}}{1 + v^\top S^{-1} v}
      \, , \quad \text{so that}
    \\
    \langle (S + v v^\top)^{-1} v, v\rangle
    &= \langle S^{-1} v, v \rangle - \frac{\langle S^{-1} v, v \rangle^2}{1 + \langle S^{-1} v, v \rangle} 
    = \frac{\langle S^{-1} v, v \rangle}{1 + \langle S^{-1} v, v \rangle}
    ,
  \end{align*}
  hence $\langle (S + vv^\top)^{-1} v, v\rangle \in [0, 1)$.
  Inverting this equality yields~\eqref{eq:sherman-morrison}.
\end{proof}

We now deduce from
Theorem~\ref{thm:minimax-leverage-score} a precise lower bound on the minimax risk~\eqref{eq:minimax-risk-ls}, valid for every distribution of covariates $P_X$.
By Proposition~\ref{prop:degenerate-minimax-risk}, it suffices to consider the case when $n \geq d$ and $P_X$ is nondegenerate
(since otherwise the minimax risk is infinite).

\begin{corollary}[Minimax lower bound]
  \label{cor:minimax-lower-bound}
  Under the assumptions of Theorem~\ref{thm:ls-minimax}, the minimax risk~\eqref{eq:minimax-risk-ls} over $\gaussclass (P_X, \sigma^2)$
  satisfies
  \begin{equation}
    \label{eq:minimax-lower-bound}
    \inf_{\wh \beta_n} \sup_{P \in \gaussclass (P_X, \sigma^2)} \E
    \big[ \excessrisk_P (\wh \beta_n) \big]
    \geq 
    \frac{\sigma^2 d}{n - d + 1}
    \, .
  \end{equation}
\end{corollary}

\begin{proof}[Proof of Corollary~\ref{cor:minimax-lower-bound}]
  By Theorem~\ref{thm:minimax-leverage-score}, the minimax excess risk over $\gaussclass (P_X, \sigma^2)$ writes:
  \begin{equation}
    \label{eq:proof-minimax-lower-1}
    \sigma^2 \cdot \E \bigg[ \frac{\wh \ell_{n+1}}{1 - \wh \ell_{n+1}} \bigg]
    \geq \sigma^2 \cdot \frac{\E [ \wh \ell_{n+1} ]}{1 - \E [ \wh \ell_{n+1} ]}
    \, ,
  \end{equation}
  where the inequality follows from the convexity of the map $x \mapsto x / (1-x) = 1 - 1 / (1-x)$ on $[0, 1)$.
  Now, by exchangeability of $(X_1, \dots, X_{n+1})$,
  \begin{align}
    \E [ \wh \ell_{n+1} ]
    &= \frac{1}{n+1} \sum_{i=1}^{n+1} \E \bigg[ \bigg\langle \bigg( \sum_{i=1}^{n+1} X_i X_i^\top \bigg)^{-1} X_{i}, X_{i} \bigg\rangle \bigg] \nonumber \\
    &= \frac{1}{n+1} \E \bigg[ \tr \bigg\{ \bigg( \sum_{i=1}^{n+1} X_i X_i^\top \bigg)^{-1} \bigg( \sum_{i=1}^{n+1} X_i X_i^\top \bigg) \bigg\} \bigg]
      = \frac{d}{n+1}
      \label{eq:proof-minimax-lower-2}
      \, .
  \end{align}
  Plugging equation~\eqref{eq:proof-minimax-lower-2} into~\eqref{eq:proof-minimax-lower-1} yields the lower bound~\eqref{eq:minimax-lower-bound}.
\end{proof}

Since $n - d + 1 \leq n$, Corollary~\ref{cor:minimax-lower-bound} implies a lower bound of $\sigma^2 d / n$.
The minimax risk for linear regression has been determined under additional boundedness assumptions on $Y$ (and thus on $\beta^*$) by \cite{tsybakov2003optimal}, showing that it scales as $\Theta (d/n)$ up to numerical constants.
The proof of the lower bound relies on information-theoretic arguments, and in particular on Fano's inequality \cite{tsybakov2009nonparametric}.
Although widely applicable, such techniques often lead to loose constant factors.
By contrast, the approach relying on Bayes risk
leading to Corollary~\ref{cor:minimax-lower-bound} recovers the optimal leading constant, owing to the analytical tractability of the problem.

In fact, the lower bound of Corollary~\ref{cor:minimax-lower-bound} is more precise than the $\sigma^2 d/n$ lower bound, in particular when the dimension $d$ is commensurate to $n$.
Indeed, in the case of centered Gaussian design, namely when $X \sim \gaussdist (0, \Sigma)$ for some positive matrix $\Sigma$, the risk of the OLS estimator (and thus, by Theorem~\ref{thm:ls-minimax}, the minimax risk)
can be computed exactly \cite{anderson2003introduction,breiman1983many}, and equals
  \begin{equation}
    \label{eq:risk-gaussian-design}
    \E \big[ \excessrisk_P (\lsb) \big]
    = \frac{\sigma^2 d}{n - d - 1}
    \, .
  \end{equation}
  The distribution-independent lower bound of Corollary~\ref{cor:minimax-lower-bound} is very close to the above whenever $n - d \gg 1$.
  Hence, it is almost the best possible distribution-independent lower bound on the minimax risk.
  This also shows that Gaussian design is almost the easiest design distribution in terms of minimax risk.
  This can be understood as follows: degeneracy (a large value of $\tr (\wt \Sigma_n^{-1})$) occurs whenever the rescaled sample covariance matrix $\wt \Sigma_n$ is small in some direction;
  this occurs if either the direction of $\wt X = \Sigma^{-1/2} X$ is far from uniform (so that the projection of $\wt X$ in some direction can be small), or if its norm can be small.
  If $\wt X \sim \gaussdist (0, \id)$, then $\wt X / \| \wt X \|$ is uniformly distributed on the unit sphere, while $\| \wt X \| = \sqrt{\sum_{j=1}^d (\wt X^j)^2}$ is sharply concentrated around $\sqrt{d}$: %
  with high probability, $\| \wt X \| = \sqrt{d} + O (1)$ (see \eg \cite[Eq.~3.7]{vershynin2018high}).

  In particular, in the high-dimensional regime where $d$ and $n$ are large and commensurate, namely $d, n \to \infty$ and $d/n \to \gamma$, the lower bound of Corollary~\ref{cor:minimax-lower-bound} matches the minimax risk~\eqref{eq:risk-gaussian-design} in the Gaussian case, which converges to $\sigma^2 \gamma / (1-\gamma)$.
  The limit $\sigma^2 \gamma / (1 - \gamma)$ has a form of universality in the high-dimensional regime: indeed, it is connected to the Marchenko-Pastur law for the spectrum of random matrices~\cite{marchenko1967distribution}, which extends to more general distributions with jointly independent coordinates.
  However, the ``universality'' of this limiting behavior is quite restrictive \cite{elkaroui2011geometric,elkaroui2018predictor}, since it relies on the assumption of independent covariates, which induces in high dimension a very specific geometry due to the concentration of measure phenomenon \cite{ledoux2001concentration,boucheron2013concentration}.
  For instance, \cite{elkaroui2018predictor} obtains different limiting risks for robust regression in high dimension when considering non-independent coordinates.
  Corollary~\ref{cor:minimax-lower-bound} shows that, if not universal, the limiting risk obtained in the independent case provides a \emph{lower bound} for general design distributions.  

  Finally, the property of the design distribution that leads to the minimal excess risk in high dimension can be formulated succinctly in terms of leverage scores, using Theorem~\ref{thm:minimax-leverage-score}.

  \begin{corollary}
    \label{cor:leverage-score-constant}
    Let $(d_n)_{n \geq 1}$ be a sequence of positive integers such that $d_n / n \to \gamma \in (0, 1)$, and $(P_X^{(n)})_{n \geq 1}$ a sequence of non-degenerate distributions on $\R^{d_n}$.
    Assume that the minimax excess risk~\eqref{eq:minimax-risk-ls} over
    $\wellclass (P_X^{(n)}, \sigma^2)$ converges to $\sigma^2 \gamma / (1 - \gamma)$.
    Then, the distribution of the leverage score $\wh \ell_{n+1}^{(n)}$ of one sample among $n+1$ under $P_X^{(n)}$ converges in probability to $\gamma$.
  \end{corollary}

  \begin{proof}
    Let $\phi (x) = x / (1 - x)$ for $x \in [0, 1)$, and $\psi (x) := \phi (x) - \phi (\gamma) - \phi' (\gamma) (x - \gamma)$ (with $\psi (\gamma) = 0$).
    Since $\phi$ is strictly convex, $\psi (x) > 0$ for $x \neq \gamma$, and $\psi$ is also strictly convex.
    Hence, $\psi$ is decreasing on $[0, \gamma]$ and increasing on $[\gamma, 1)$.
    In particular, for every $\eps > 0$, $\eta_\eps := \inf_{|x - \gamma | \geq \eps} \psi (x) > 0$.
    
    By Theorem~\ref{thm:minimax-leverage-score}, the assumption of Corollary~\ref{cor:leverage-score-constant} means that $\E [ \phi (\wh \ell_{n+1}^{(n)}) ] \to \phi (\gamma)$.
    Since in addition $\E [ \wh \ell_{n+1}^{(n)} ] = d_n / (n+1) \to \gamma$ (the first equality, used in the proof of Corollary~\ref{cor:minimax-lower-bound}, holds for $d_n \leq n + 1$, hence for $n$ large enough since $\gamma < 1$), we have $\E [ \psi (\wh \ell_{n+1}^{(n)}) ] \to 0$.
    Now, for every $\eps > 0$, $\psi (x) \geq \eta_\eps \cdot \indic{| x - \gamma | \geq \eps}$, so that $\P ( | \wh \ell_{n+1}^{(n)} - \gamma | \geq \eps ) \leq \eta_\eps^{-1} \E [ \psi (\wh \ell_{n+1}^{(n)}) ] \to 0$.
  \end{proof}

  \subsection{Upper bounds on the minimax risk%
  }
\label{sec:upper-bound-minimax}

In this section, we complement the lower bound of Corollary~\ref{cor:minimax-lower-bound} by providing matching \emph{upper bounds} on the minimax risk.
Since by Proposition~\ref{prop:degenerate-minimax-risk} the minimax risk is infinite when the design distribution is degenerate, we introduce the following quantitative version of the non-degeneracy condition:

\begin{assumption}[Small-ball condition]
  \label{ass:small-ball}
  The whitened design $\wt X = \Sigma^{-1/2} X$ satisfies the following: there exist constants $C \geq 1$ and $\alpha \in (0, 1]$ such that, for every
  linear hyperplane $H$ of $\R^d$ and $t > 0$,
  \begin{equation}
    \label{eq:small-ball}
    \P \big( \dist (\wt X, H) \leq t \big)
    \leq (C t)^\alpha
    \, .
  \end{equation}
  Equivalently, for every $\theta \in \R^d \setminus \{ 0 \}$ and $t > 0$,
  \begin{equation}
    \label{eq:small-ball-gen}
    \P \big( | \langle \theta, X\rangle | \leq t \| \theta \|_{\Sigma} \big)
    \leq (C t)^\alpha
    \, .
  \end{equation}
\end{assumption}

The equivalence between~\eqref{eq:small-ball} and~\eqref{eq:small-ball-gen} comes from the fact that the distance $\dist (\wt X, H)$ of $\wt X$ to the hyperplane $H$ equals $| \langle \theta', \wt X\rangle |$, where $\theta' \in S^{d-1}$ is a normal vector to $H$.
Condition~\eqref{eq:small-ball-gen} is then recovered by letting $\theta = \Sigma^{-1/2} \theta'$ (such that $\| \theta \|_\Sigma = \| \theta' \| = 1$)
and by homogeneity.

Assumption~\ref{ass:small-ball} states that $\wt X$ does not lie too close to any fixed hyperplane.
This assumption is a strengthened variant of the ``small ball'' condition introduced by \cite{koltchinskii2015smallest,mendelson2015smallball,lecue2016performance} in the analysis of sample covariance matrices and least squares regression, which amounts to assuming~\eqref{eq:small-ball-gen} for a {single} value of $t < C^{-1}$.
This latter condition amounts to a uniform equivalence between the $L^1$ and $L^2$ norms of one-dimensional marginals $\langle \theta, X\rangle$ ($\theta \in \R^d$) of $X$ \cite{koltchinskii2015smallest}.
Here, we require that the condition holds for arbitrarily small $t$; the reason for this is that in order to control the minimax excess risk~\eqref{eq:minimax-risk-ls} (and thus $\E [ \tr (\wt \Sigma_n^{-1}) ]$), we are led to control the lower tail of the rescaled covariance matrix $\wt \Sigma_n$ at all confidence levels.
The study of the lower tail of $\wt \Sigma_n$ (on which the results of this section rely) is deferred to Section~\ref{sec:small-ball-bounds}.
We also illustrate Assumption~\ref{ass:small-ball} in Section~\ref{sec:small-ball-indep}, by discussing conditions under which it holds in the case of independent coordinates.

First, Assumption~\ref{ass:small-ball} itself suffices to obtain an upper bound on the minimax risk of $O (\sigma^2 d / n)$, without %
additional assumptions on the upper tail of $X X^\top$ (apart from integrability).

\begin{proposition}
  \label{prop:upper-well-constant}
  If Assumption~\ref{ass:small-ball} holds, then for every $P \in \wellclass (P_X, \sigma^2)$, letting $C' = 3 C^4 e^{1+9/\alpha}$ we have:
  \begin{equation}
    \label{eq:upper-well-constant}
    \E [ \excessrisk (\lsb) ]
    \leq 2 C' \cdot \frac{\sigma^2 d}{n}
    \, .
  \end{equation}
\end{proposition}

Proposition~\ref{prop:upper-well-constant} (a consequence of Corollary~\ref{cor:negative-moments-smallest} from Section~\ref{sec:upper-bound-lower}) is
optimal in terms of the rate of convergence; however, it exhibits the suboptimal $2 C'$ factor in the leading term.
As we show next, it is possible to obtain an optimal constant in the first-order term (as well as a second-order term of the correct order) under a modest additional assumption.

\begin{assumption}[Norm kurtosis]
  \label{ass:fourth-moment-norm}
  $\E [ \| \Sigma^{-1/2} X \|^4 ] \leq \kappa d^2$ for some $\kappa > 0$.
\end{assumption}

\begin{remark}
  \label{rem:fourth-moment-assumptions}
  Since $\E [ \| \Sigma^{-1/2} X \|^2 ] = d$,
  Assumption~\ref{ass:fourth-moment-norm}
  is a bound on the kurtosis of the variable $\| \Sigma^{-1/2} X \|$.
  This condition is implied by the following $L^2$-$L^4$ equivalence for one-dimensional marginals of $X$: for every $\theta \in \R^d$, $\E [ \langle \theta, X\rangle^4 ]^{1/4} \leq \kappa^{1/4} \cdot \E [ \langle \theta, X\rangle^2]^{1/2}$ (Assumption~\ref{ass:equiv-l4-l2} below).
  Indeed, assuming that the latter holds, then taking $\theta = \Sigma^{-1/2} e_j$ (where $(e_j)_{1 \leq j \leq d}$ denotes the canonical basis of $\R^d$), so that $\langle \theta, X\rangle$ is the $j$-th coordinate $\wt X^j$ of $\wt X$, we get $\E [ (\wt X^j)^4 ] \leq \kappa \wt \E [ (\wt X^j)^2 ]^2 = \kappa$ (since $\E [\wt X \wt X^\top] = I_d$).
  This implies that
  \begin{align*}
    \E \big[ \| \wt X \|^4 \big]
    &= \E \bigg[ \bigg( \sum_{j=1}^d (\wt X^j)^2 \bigg)^2 \bigg]
    = \sum_{1\leq j, k \leq d} \E \big[ (\wt X^j)^2 (\wt X^k)^2 \big] \\
    &\leq \sum_{1\leq j, k \leq d} \E \big[ (\wt X^j)^4 \big]^{1/2} \E \big[ (\wt X^k)^4 \big]^{1/2} 
    \leq \sum_{1\leq j, k \leq d} \kappa^{1/2} \cdot \kappa^{1/2}
      = \kappa \cdot d^2
      \, ,
  \end{align*}
  where the first inequality above comes from the Cauchy-Schwarz inequality.
  The converse is false: if $\wt X$ is uniform on $\{ \sqrt{d} e_j : 1 \leq j \leq d \}$, then the first condition holds with $\kappa = 1$, while the second only holds for $\kappa \geq d$ (taking $\theta = e_1$).
  Hence, Assumption~\ref{ass:fourth-moment-norm} on the upper tail of $X$ is weaker than an $L^2$-$L^4$ equivalence of the one-dimensional marginals of $X$; on the other hand, we do require a small-ball condition (Assumption~\ref{ass:small-ball}) on the lower tail of $X$.
\end{remark}

\begin{theorem}[Upper bound in the well-specified case%
  ]
  \label{thm:upper-well-spec}
  Grant Assumptions~\ref{ass:small-ball} and~\ref{ass:fourth-moment-norm}.
  Let %
  $C' = 3 C^4 e^{1+9/\alpha}$ \textup(which only depends on $\alpha, C$\textup).
  If $n \geq \min (6 \alpha^{-1} d,  12 \alpha^{-1} \log (12 \alpha^{-1}) )$, then 
  \begin{equation}
    \label{eq:expected-inverse-trace}
    \frac{1}{n} \E \big[ \tr ( \wt \Sigma_n^{-1} ) \big]
    \leq \frac{d}{n} + 8 C' \kappa \Big( \frac{d}{n} \Big)^2
    \, .
  \end{equation}
  In particular, the minimax excess risk
  over the class $\wellclass (P_X, \sigma^2)$
  satisfies:
  \begin{equation}
    \label{eq:upper-bound-minimax}
    \frac{\sigma^2d}{n}
    \leq \inf_{\wh \beta_n} \sup_{P \in \wellclass (P_X, \sigma^2)} \E
    \big[ \excessrisk_P (\wh \beta_n) \big]
    \leq \frac{\sigma^2d}{n} \Big( 1 + 8 C' \frac{\kappa d}{n} \Big)
    \, .
  \end{equation}  
\end{theorem}

The proof of Theorem~\ref{thm:upper-well-spec} is given in Section~\ref{sec:proof-trace-inverse};
it relies in particular on Lemma~\ref{lem:trace-inverse-approx} herein and on Theorem~\ref{thm:lower-tail-covariance} from Section~\ref{sec:small-ball-bounds}.
From a technical point of view, some care is required since the assumptions of Theorem~\ref{thm:upper-well-spec} provide control on lower, rather than upper, relative deviations of $\wh \Sigma_n$ with respect to $\Sigma$.
As shown by the lower bound (established in Corollary~\ref{cor:minimax-lower-bound}), the constant in the first-order term in~\eqref{eq:upper-bound-minimax} is tight; in addition, one could see from a higher-order expansion (under additional moment assumptions) that the second-order term is also tight, up to the constant $8 C'$ factor.

Consider now the general misspecified case, namely the class $\misclass (P_X, \sigma^2)$.
Here, we will need the slightly stronger Assumption~\ref{ass:equiv-l4-l2}.

\begin{assumption}[$L^2$-$L^4$ norm equivalence]%
  \label{ass:equiv-l4-l2}
  There exists a constant $\kappa > 0$ such that, for every $\theta \in \R^d$,
  $\E [ \langle \theta, X\rangle^4 ] \leq \kappa \cdot \E [ \langle \theta, X\rangle^2 ]^2$.
\end{assumption}

\begin{proposition}[Upper bound in the misspecified case]
  \label{prop:upper-misspecified}
  Assume that $P_X$ satisfies Assumptions~\ref{ass:small-ball} and~\ref{ass:equiv-l4-l2}, and that
  \begin{equation*}
    \chi := \E \big[ \E [ \eps^2 %
    \cond X ]^2 \| \Sigma^{-1/2} X \|^4 \big] / d^2 < + \infty
  \end{equation*}
  \textup(note that $\chi \leq \E [ ( Y - \langle \beta^*, X\rangle )^4 \| \Sigma^{-1/2} X \|^4 ]/d^2$\textup).
  Then, {for $n \geq \max (96, 6 d) / \alpha$}, the risk of the OLS estimator satisfies
  \begin{equation}
    \label{eq:upper-misspecified}
    \E \big[ \excessrisk (\lsb) \big]
    \leq \frac{1}{n} \E \big[ (Y - \langle \beta^*, X\rangle)^2 \| \Sigma^{-1/2} X \|^2 \big] + 276 C'^2 \sqrt{\kappa \chi} \Big( \frac{d}{n} \Big)^{3/2}
    \, .
  \end{equation}
  In particular,
  we have
  \begin{equation}
    \label{eq:upper-misspecified-minimax}
    \frac{\sigma^2 d}{n}
    \leq
    \inf_{\wh \beta_n} \sup_{P \in \misclass (P_X, \sigma^2)}
    \E \big[ \excessrisk (\wh \beta_n) \big]
    \leq \frac{\sigma^2 d}{n} \bigg( 1 + 276 C'^2 \kappa \sqrt{\frac{d}{n}} \bigg)
    \, .
  \end{equation}
\end{proposition}

The proof of Proposition~\ref{prop:upper-misspecified} is provided in Section~\ref{sec:proof-upper-misspecified}; it combines results from Section~\ref{sec:small-ball-bounds} with a tail bound from \cite{oliveira2016covariance}.
Proposition~\ref{prop:upper-misspecified} shows that, under Assumptions~\ref{ass:small-ball} and~\ref{ass:equiv-l4-l2}, the minimax excess risk over the class $\misclass (P_X, \sigma^2)$ scales as $(1 + o (1)) \sigma^2 d / n$ as $d/n \to 0$.
This implies that the OLS estimator is asymptotically minimax on the misspecified class $\misclass (P_X, \sigma^2)$ when $d = o (n)$, and that independent Gaussian noise is asymptotically the least favorable
structure for the error~$\eps$.

\subsection{Parameter estimation}
\label{sec:parameter-estimation}

Let us
briefly discuss
how the results of this section obtained for prediction can be adapted to the problem of parameter estimation, where the loss of an estimate $\wh \beta_n$ given $\beta^*$ is $\| \wh \beta_n - \beta^* \|^2 $.

By the same proof as that of Theorem~\ref{thm:ls-minimax} (replacing the norm $\| \cdot \|_\Sigma$ by $\| \cdot \|$), the minimax excess risk over the classes $\gaussclass (P_X, \sigma^2)$ and $\wellclass (P_X, \sigma^2)$ is $(\sigma^2 / n) \E [ \tr (\wh \Sigma_n^{-1})]$, achieved by the OLS estimator.
By convexity of $A \mapsto \tr (A^{-1})$ over positive matrices \cite{lowner1934monotone}, this quantity is larger than $\sigma^2 \tr (\Sigma^{-1}) / n$.

In the case of centered Gaussian covariates, $\E [ \tr (\wh \Sigma_n^{-1})] = \tr ( \Sigma^{-1} \E [ \wt \Sigma_n^{-1} ] ) = \tr (\Sigma^{-1}) n / (n - d - 1)$ \cite{anderson2003introduction},
so the minimax risk is $\sigma^2 \tr (\Sigma^{-1}) / (n - d - 1)$.
On the other hand, the improved lower bound for general design of Corollary~\ref{cor:minimax-lower-bound} for prediction does not appear to extend to estimation.
The reason for this is that the map $A \mapsto A / (1 - \tr (A))$ is not convex over positive matrices for $d \geq 2$ (where convexity is defined with respect to the positive definite order, see \eg \cite[Section~3.6.2]{boyd2004convex}), although its trace is.

Finally, the results of Section~\ref{sec:small-ball-bounds} on the lower tail of $\wt \Sigma_n$ can be used to obtain upper bounds in a similar fashion as for prediction.
For instance, an analogue of Proposition~\ref{prop:upper-well-constant} can be directly obtained by bounding $\tr (\wh \Sigma_n^{-1}) \leq \lamin (\wt \Sigma_n)^{-1} \cdot \tr (\Sigma^{-1})$.
Since this work is primarily focused on prediction,
we do not elaborate further in this direction.

\section{Bounding the lower tail of a sample covariance matrix at all probability levels}
\label{sec:small-ball-bounds}

Throughout this section, up to replacing $X$ by $\Sigma^{-1/2} X$, we assume unless otherwise stated that $\E [ X X^\top] = I_d$.
Our aim is to obtain non-asymptotic large deviation inequalities of the form:
\begin{equation*}
  \P ( \lamin (\wh \Sigma_n) \leq t)
  \leq e^{- n \psi (t)}
\end{equation*}
where $\psi (t) \to \infty$ as $t \to 0^+$.
Existing bounds \cite{vershynin2012introduction,srivastava2013covariance,koltchinskii2015smallest,oliveira2016covariance} are typically sub-Gaussian bounds with $\psi (t) = c (1 - C \sqrt{d/n} - t)_+^2$ for some constants $c, C > 0$, which ``saturate''
for small $t$.
In this section, we %
study
the behavior of the large deviations for small values of $t$,
namely $t \in (0, c)$, where $c < 1$ is a fixed constant.
In Section~\ref{sec:lower-bound-lower-tail}, we provide a lower bound on these tail probabilities, namely an upper bound on $\psi$, valid for every distribution of $X$ when $d \geq 2$.
In Section~\ref{sec:upper-bound-lower}, we show that Assumption~\ref{ass:small-ball} is necessary and sufficient to obtain tail bounds of the optimal order.
Finally, in Section~\ref{sec:small-ball-indep} we show that Assumption~\ref{ass:small-ball} is naturally satisfied in the case of independent coordinates, under a mild regularity condition on their distributions.

\subsection{A general lower bound on the lower tail%
}
\label{sec:lower-bound-lower-tail}

First, Proposition~\ref{prop:small-ball-lower} below shows that in dimension $d \geq 2$, the probability of deviations of $\lamin(\wh \Sigma_n)$ cannot be arbitrarily small.

\begin{proposition}
  \label{prop:small-ball-lower}
  Assume that $d \geq 2$.
  Let $X$ be a random vector in $\R^d$ such that $\E [ X X^\top ] = \id$.
  Then, for every $t \leq 1$, 
  \begin{equation}
    \label{eq:small-ball-lower-1}
    \sup_{\theta \in S^{d-1}} \P ( | \langle \theta, X \rangle | \leq t )
    \geq 0.16 \cdot t
    \, ,
  \end{equation}
  and
  therefore
  \begin{equation}
    \label{eq:small-ball-lower-2}
    \P \big( \lamin (\wh \Sigma_n) \leq t \big)
    \geq (0.025 \cdot t)^{n/2}
    \, .
  \end{equation}
\end{proposition}

The assumption that $d \geq 2$ is necessary since for $d=1$, if $X = 1$ almost surely, then $\lamin (\wh \Sigma_n) = 1$ almost surely.
Proposition~\ref{prop:small-ball-lower} is proved
in Section~\ref{sec:proof-lower-lower}
through a probabilistic argument, namely by considering a random vector $\theta$ drawn uniformly on the sphere $S^{d-1}$.

Proposition~\ref{prop:small-ball-lower} shows that $\P (\lamin (\wh \Sigma_n) \leq t)$ is at least $(C t)^{c n}$, where $C = 0.025$ and $c= 1/2$ are absolute constants; this bound writes $e^{-n \psi (t)}$, where $\psi (t) \asymp \log (1/t)$ as $t \to 0^+$.
In the following section, we study matching upper bounds on this lower tail.

\subsection{Optimal control of the lower tail%
}
\label{sec:upper-bound-lower}

In this section, we study conditions under which an upper bound matching the lower bound from Proposition~\ref{prop:small-ball-lower} can be obtained.
We start by noting that Assumption~\ref{ass:small-ball} is necessary to obtain such bounds:

\begin{remark}[Necessity of small ball condition]
  \label{rem:necessary-small-ball}
  Assume that there exists $c_1, c_2 > 0$ such that $\P (\lamin (\wh \Sigma_n) \leq t) \leq (c_1 t)^{c_2 n}$ for all $t \in (0, 1)$.
  Then, Lemma~\ref{lem:smallest-eigenvalue-min} below implies that $\sup_{\theta \in S^{d-1}} \P ( |\langle \theta, X\rangle | \leq t ) \leq (c_1 t^2)^{c_2}$ for all $t \in (0, 1)$.
  Hence,
  $P_X$ satisfies Assumption~\ref{ass:small-ball} with $C = \sqrt{c_1}$ and $\alpha = 2 c_2$.
\end{remark}

\begin{lemma}
  \label{lem:smallest-eigenvalue-min}
  For $t \in (0, 1)$, let $p_t = \sup_{\theta \in S^{d-1}} \P ( |\langle \theta, X\rangle | \leq t )$.
  Then, $\P (\lamin (\wh \Sigma_n) \leq t) \geq p_{\sqrt{t}}^n$.
\end{lemma}

\begin{proof}[Proof of Lemma~\ref{lem:smallest-eigenvalue-min}]
  Let $p < p_{\sqrt{t}}$.
  By definition of $p_{\sqrt{t}}$, there exists $\theta \in S^{d-1}$ such that $\P (\langle \theta, X\rangle^2 \leq t) \geq p$.
  Hence, by independence, with probability at least $p^n$, $\langle \theta, X_i\rangle^2 \leq t$ for $i = 1, \dots, n$, so that $\lamin (\wh \Sigma_n) \leq \langle \wh \Sigma_n \theta, \theta \rangle \leq t$.
  Taking $p \to p_{\sqrt{t}}$ concludes the proof.
\end{proof}

As Theorem~\ref{thm:lower-tail-covariance} shows, Assumption~\ref{ass:small-ball} is also sufficient to obtain an optimal control on the lower tail.

\begin{theorem}
  \label{thm:lower-tail-covariance}
  Let $X$ be a random vector in $\R^d$.
  Assume that $\E [X X^\top] = \id$ and that $X$ satisfies Assumption~\ref{ass:small-ball}.
  If $n \geq 6 d / \alpha$,
  then for every $t \in (0, 1)$\textup:
  \begin{equation}
    \label{eq:lower-tail-covariance}
    \P \big( \lamin (\wh \Sigma_n) \leq t \big)
    \leq (C' t)^{\alpha n/6}
  \end{equation}
  where $C' = 3 C^4 e^{1+9/\alpha}$.
\end{theorem}

Theorem~\ref{thm:lower-tail-covariance} can be stated in the non-isotropic case, where $\Sigma = \E [X X^\top]$ is arbitrary:

\begin{corollary}
  \label{cor:lower-tail-noniso}
  Let $X$ be a random vector in $\R^d$ such that $\E [ \| X \|^2] < + \infty$, and let $\Sigma = \E [X X^\top]$.
  Assume that $X$ satisfies Assumption~\ref{ass:small-ball}.
  Then, if $d/n \leq \alpha/6$, for every $t \in (0, 1)$, the empirical covariance matrix $\wh \Sigma_n$ formed with an \iid sample of size $n$ satisfies
  \begin{equation}
    \label{eq:lowertail-noniso}
    \wh \Sigma_n \mgeq t \Sigma
  \end{equation}
  with probability at least $1 - (C' t)^{\alpha n /6}$, 
  where $C'$ is as in Theorem~\ref{thm:lower-tail-covariance}.
\end{corollary}

\begin{proof}[Proof of Corollary~\ref{cor:lower-tail-noniso}]
  We may assume that $\Sigma$ is invertible: otherwise, we can just consider the span of the support of $X$, a subspace of $\R^d$ of dimension $d' \leq d \leq \alpha n / 6$.
  Now, let $\wt X = \Sigma^{-1/2} X$; by definition, $\E [\wt X \wt X^\top] = \id$, and $\wt X$ satisfies Assumption~\ref{ass:small-ball} since $X$ does.
  By Theorem~\ref{thm:lower-tail-covariance}, with probability at least $1 - (C' t)^{\alpha n/6}$, $\lamin (\Sigma^{-1/2} \wh \Sigma_n \Sigma^{-1/2}) \geq t$, which amounts to $\Sigma^{-1/2} \wh \Sigma_n \Sigma^{-1/2} \mgeq t \id$, and thus $\wh \Sigma_n \mgeq t \Sigma$.
\end{proof}

It is worth noting that Theorem~\ref{thm:lower-tail-covariance} does not require any condition on the upper tail of $X X^\top$, aside from the
assumption $\E [X X^\top] = I_d$.
Indeed, as noted in Remark~\ref{rem:necessary-small-ball}, it only requires the necessary Assumption~\ref{ass:small-ball}.
In particular, it does not require any sub-Gaussian assumption on $X$,
similarly to the results from \cite{koltchinskii2015smallest,oliveira2016covariance,vandegeer2014higher,yaskov2014lower,yaskov2015sharp};
this owes to the fact that lower bounds for sums of non-negative random variables hold under weak assumptions.

\begin{remark}[Extension to random quadratic forms]
  Theorem~\ref{thm:lower-tail-covariance} extends (up to
  straightforward changes in notations) to random quadratic forms $v \mapsto \langle A_i v, v\rangle$ where $A_1, \dots, A_n$ are positive semi-definite and \iid, with $\E [A_i] = I_d$ (Theorem~\ref{thm:lower-tail-covariance} corresponds to the rank $1$ case where $A_i = X_i X_i^\top$).
  On the other hand, the lower bound of Proposition~\ref{prop:small-ball-lower} is specific to rank $1$ matrices,
  as can be seen by considering the counterexample where $A_i = I_d$ almost surely.
\end{remark}

\begin{remark}[Gaussian case]
  It may be worth comparing the bound~\eqref{eq:lower-tail-covariance} to known estimates in the special case of the Gaussian distribution, namely $X \sim \gaussdist (0, I_d)$.
  In this case, the joint density of eigenvalues of $\wh \Sigma_n$ admits a closed-form expression,
  which provides by marginalization the density of $\lamin (\wh \Sigma_n)$ \cite[p. 533]{edelman1988eigenvalues}.
  From this expression, the following bound is deduced in~\cite[eq.~(99)]{wu2012optimal}:
  \begin{equation*}
    \P \Big( \lamin (\wh \Sigma_n) \leq t \Big)
    \leq \frac{2 (n/2)^{(n - d + 1)/2}}{n - d + 1} \frac{\sqrt{\pi} \Gamma (\frac{n+1}{2})}{\Gamma (\frac{d}{2}) \Gamma (\frac{n - d + 1}{2}) \Gamma (\frac{n - d + 2}{2})} t^{n - d + 1}
    \, .
  \end{equation*}
  Letting $d = d_n$ such that $d_n / n \to \alpha \in (0, 1)$ and applying Stirling's approximation, this implies the following large deviation estimate~\cite[Lemma~1]{wu2012optimal}: for any fixed $t \in (0, 1)$,
  \begin{align*}
    \P \Big( \lamin (\wh \Sigma_n) \leq t \Big)
    &\leq
    \Big( \frac{n}{d} \Big)^{d/2} \bigg( \frac{\sqrt{e} t}{1 - d/n} \bigg)^{n - d + o (n)}
    \, .
  \end{align*}
  The bound~\eqref{eq:lower-tail-covariance} is of this form; it holds for general distributions of $X$, at the cost of worst constants in the Gaussian case.
\end{remark}

\paragraph{Idea of the proof.}

The proof of Theorem~\ref{thm:lower-tail-covariance} is provided in Section~\ref{sec:proof-lower-tail}.
It builds on the analysis of \cite{oliveira2016covariance}, who obtains sub-Gaussian deviation bounds under fourth moment assumptions (Assumption~\ref{ass:equiv-l4-l2}), although some refinements are needed to handle our considered regime (with $t$ small).
We now discuss some general ideas about the proof technique.

The proof starts with the representation of $\lamin (\wh \Sigma_n)$ as the infimum of an empirical process:
\begin{equation}
  \label{eq:lamin-inf-process}
  \lamin (\wh \Sigma_n)
  = \inf_{\theta \in S^{d-1}} \langle \wh \Sigma_n \theta, \theta\rangle
  = \inf_{\theta \in S^{d-1}} \bigg\{ Z (\theta) := \frac{1}{n} \sum_{i=1}^n \langle \theta, X_i\rangle^2 \bigg\}
  \, .
\end{equation}
In order to control this infimum, a natural approach is to first control $Z (\theta)$ on a suitable finite $\eps$-covering of $S^{d-1}$ using Assumption~\ref{ass:small-ball}, independence, and a union bound, and then to extend this control to $S^{d-1}$ by approximation.
However, this approach (see \eg \cite[Theorem~5.39]{vershynin2012introduction} for a use of this argument) fails here, since the control of the approximation term would require an exponential upper bound on $\opnorm{\wh \Sigma_n}$,
which does not hold for heavy-tailed distributions.
Instead, as in \cite{oliveira2016covariance}, we use the so-called PAC-Bayesian inequality for empirical processes \cite{mcallester1999pacbayesml,mcallester2003pac,langford2003pac,catoni2007pacbayes,audibert2011robust}, which is based on a variational representation of the relative entropy.
This technique enables one to control a smoothed version of the process $Z (\theta)$, namely
\begin{equation*}
  Z (\rho)
  :=
  \int_{\R^d} Z (\theta) \rho (\di \theta)
  \, ,
\end{equation*}
indexed by probability distributions $\rho$ on $\Theta$.

Specifically, let $\pi$ be a probability distribution on some subset $\Theta \subset \R^d$ containing $S^{d-1}$.
In addition, let $\psi : \R_+^* \to \R$ be a bound on the moment generating function of $- \langle \theta, X\rangle^2$, such that for all $\lambda > 0$ and $\theta \in \Theta$,
\begin{equation*}
  \E \exp \big(- \lambda \langle \theta, X\rangle^2 \big)
  \leq e^{- \psi (\lambda)}, %
  \quad \text{so that} \quad
  \E \exp \big( - \lambda n Z (\theta) - n \psi (\lambda) \big)
  \leq 1
  .
\end{equation*}
The PAC-Bayes variational inequality (see Lemma~\ref{lem:pac-bayes} for a general statement) allows to turn this (pointwise, for every $\theta$) bound on the moment generating function into a uniform bound for the smoothed process:
for every $t > 0$,
\begin{equation*}
  \P \left( \forall \rho, \ 
   - \lambda n \big[ Z (\rho) + \psi (\lambda) \big] 
    \leq %
    \kll{\rho}{\pi} + t
  \right)
  \geq 1 - e^{-t}
  \, ,
\end{equation*}
where $\rho$ spans all distributions over $\Theta$ and $\kll{\rho}{\pi} = \int \log \frac{\di \rho}{\di \pi} \di \rho$ is the relative entropy between $\rho$ and $\pi$.
One then deduce from these inequalities the following decomposition.
To each $\theta \in S^{d-1}$, we associate a smoothing distribution $\rho_\theta$ around $\theta$; then, with probability at least $1 - e^{-t}$, for every $\theta \in S^{d-1}$,
\begin{align*}
  Z (\theta)
  &= Z (\theta) - \int_{\Theta} Z (\theta') \rho_{\theta} (\di \theta') + \int_{\Theta} Z (\theta') \rho_{\theta} (\di \theta') \\
  &\geq \underbrace{Z (\theta) - \int_{\Theta} Z (\theta') \rho_{\theta} (\di \theta')}_{\text{approximation term}} - \underbrace{\frac{\kll{\rho_\theta}{\pi}}{\lambda n}}_{\text{entropy term}} - \frac{\psi (\lambda) + t}{\lambda n}
    \, .
\end{align*}
The proof then involves controlling (i) the %
Laplace transform of the process;
(ii) the approximation term; and (iii) the entropy term.
In order to control the last two, a careful choice of smoothing distribution (and prior) is needed.

\begin{remark}[PAC-Bayes vs.~$\eps$-net argument]
  As indicated above, the use of an $\eps$-net argument would fail here, since it would lead to an approximation term depending on $\opnorm{\wh \Sigma_n}$.
  On the other hand, the use of a smoothing distribution which is ``isotropic'' and centered at a point $\theta$ enables one to obtain an approximation term in terms of $\tr (\wh \Sigma_n) / d
  $, which can be bounded after proper truncation of $X$ (in a way that does not overly degrade Assumption~\ref{ass:small-ball}).
\end{remark}

\begin{remark}[Choice of prior and posteriors: entropy term]
  The PAC-Bayesian technique is classically employed in conjunction with Gaussian prior and smoothing distribution \cite{langford2003pac,audibert2011robust,oliveira2016covariance}.
  This choice is convenient, since both the approximation and entropy term have closed-form expressions (in addition, a Gaussian distribution centered at $\theta$ yields the desired ``isotropic'' approximation term).
  
  However, in order to obtain non-vacuous bounds for small $t$, we need the approximation term (and thus the ``radius'' $\gamma$ of the smoothing distribution) to be small.
  But as $\gamma \to 0$, the entropy term for Gaussian distributions grows rapidly (as $d/\gamma^2$, instead of the $d \log (1/\gamma)$ rate suggested by covering numbers), which ultimately leads to vacuous bounds.
  In order to bypass this difficulty, we employ a more refined choice of prior and smoothing distributions, leading to an optimal entropy term of $d \log (1/\gamma)$.
  In addition,
  symmetry arguments show that
  this choice of smoothing also leads to an ``isotropic'' approximation term controlled by $\tr (\wh \Sigma_n) / d$
  instead of $\opnorm{\wh \Sigma_n}$.
\end{remark}

\paragraph{Formulation in terms of moments.}

The statements of this section on the lower tail of $\lamin (\wh \Sigma_n)$ can equivalently be rephrased in terms of its negative moments.
For $q \geq 1$, we denote $\| Z \|_{L^q} := \E [ |Z|^q ]^{1/q} \in [0, + \infty]$ the $L^q$ norm of a real random variable $Z$.

\begin{corollary}
  \label{cor:negative-moments-smallest}
  Under the assumptions of Theorem~\ref{thm:lower-tail-covariance} and for $n \geq 12/\alpha$, for any $1\leq q \leq \alpha n / 12$, 
  \begin{equation}
    \label{eq:negative-moment-smallest}
    \| \max (1, \lamin (\wh \Sigma_n)^{-1} ) \|_{L^q}
    \leq 2^{1/q} \cdot C'
    \, .
  \end{equation}
  Conversely, the previous inequality implies that $\P (\lamin (\wh \Sigma_n) \leq t) \leq (2 C' t)^{\alpha n /12}$ for all $t \in (0, 1)$.

  Finally, for any random vector $X$ in $\R^d$, $d\geq 2$, such that $\E [X X^\top] = I_d$, we have for any $q \geq n/2$:
  \begin{equation*}
    \| \lamin (\wh \Sigma_n)^{-1} \|_{L^q} = + \infty
    \, .
  \end{equation*}
\end{corollary}

The proof of Corollary~\ref{cor:negative-moments-smallest} is provided in Section~\ref{sec:proof-corollary-moments}.

\subsection{The small-ball condition for independent covariates}
\label{sec:small-ball-indep}

We now discuss conditions under which Assumption~\ref{ass:small-ball} holds in the case of independent coordinates.
In this section, we assume that the coordinates $X^j$, $1 \leq j \leq d$, of $X = \wt X$ are independent and centered.
Note that the condition $\E [ X X^\top] = I_d$ means that the $X^j$ have unit variance.

Let us introduce the \emph{Lévy concentration function} $Q_Z : \R^+ \to [0, 1]$ of a real random variable $Z$ defined by, for $t \geq 0$,
\begin{equation*}
  Q_Z (t) := \sup_{a \in \R} \P (|Z - a| \leq t)
  \, .
\end{equation*}
Anti-concentration (or small ball) estimates \cite{nguyen2013small} refer to nonvacuous upper bounds on this function.
Here, in order to establish Assumption~\ref{ass:small-ball}, it suffices to show that $Q_{\langle \theta, X\rangle} (t ) \leq (C t)^\alpha$ for all $t > 0$ and $\theta \in S^{d-1}$.
This amounts to establishing anti-concentration of linear combinations of independent variables $\langle \theta, X\rangle = \sum_{j=1}^d \theta^j X^j$,  uniformly over $\theta \in S^{d-1}$, namely to provide upper bounds on:
\begin{equation*}
  Q_X (t) := \sup_{\theta \in S^{d-1}} Q_{\langle \theta, X\rangle} (t)
  \, .
\end{equation*}
Small-ball probabilities naturally appear in the study of the smallest singular value of a random matrix (see \cite{rudelson2010nonasymptotic}).
\cite{tao2009littlewood,tao2009inverse,rudelson2008littlewood,rudelson2009smallest} studied anti-concentration for variables of the form $\langle \theta, X\rangle$, and deduced estimates of the smallest singular value of random matrices.
These bounds are however slightly different from the one we need:
indeed, they hold for ``unstructured'' vectors $\theta$ (which do not have additive structure, see \cite{rudelson2010nonasymptotic}), rather than uniformly over $\theta \in S^{d-1}$.
Here, in order to show that Assumption~\ref{ass:small-ball} holds, we need %
bounds over $Q_X$, which requires some assumption on the distribution of the coordinates $X^j$.

Clearly, $Q_X \geq \max_{1 \leq j \leq d} Q_{X^j}$, and in particular the coordinates $X^j$ themselves must be anti-concentrated.
Remarkably, a result of \cite{rudelson2014small} (building on a reduction by \cite{rogozin1987convolution} to uniform variables)
shows that, if the $X^j$ have bounded densities, a reverse inequality holds:

\begin{proposition}[\cite{rudelson2014small}, Theorem~1.2]
  \label{prop:rudelson-vershynin-small}
  Assume that $X^1, \dots, X^d$ are independent and have density bounded by $C_0 > 0$.
  Then, for every $\theta \in S^{d-1}$, $\sum_{j=1}^d \theta^j X^j$ has density bounded by $\sqrt{2} \, C_0$.
  In other words, $Q_X (t) \leq 2 \sqrt{2} \,  C_0 t$ for every $t > 0$, \ie, Assumption~\ref{ass:small-ball} holds with $\alpha = 1$ and $C = 2 \sqrt{2} \, C_0$.
\end{proposition}

Equivalently, if $\max_{1 \leq j \leq d} Q_{X^j} (t) \leq C t$ for all $t > 0$, then $Q_X (t) \leq \sqrt{2} C t$ for all $t > 0$, and the constant $\sqrt{2}$ is optimal \cite{rudelson2014small}.
Whether a general bound of $Q_X$ in terms of $\max_{1 \leq j \leq d} Q_{X_j}$ holds is unclear (for instance, the inequality $Q_X \leq \sqrt{2} \max_{1 \leq j \leq d} Q_{X_j}$ does not hold, as shown by considering $X^1, X^2$ independent Bernoulli $1/2$ variables, and $\theta = (1/\sqrt{2}, 1/\sqrt{2})$: then $Q_{X^j} (3/8) = 1/2$ but $Q_{\langle \theta, X\rangle} (3/8) = 3/4$).
While independence gives
\begin{equation*}
  Q_{\langle \theta, X\rangle} (t) \leq \min_{1 \leq j \leq d} Q_{\theta^j X^j} (t) = \min_{1 \leq j \leq d} Q_{X^j} (t/|\theta^j|) \leq \max_{1 \leq j \leq d} Q_{X^j} (\sqrt{d} \cdot t)
  \, ,
\end{equation*}
this bound features an undesirable dependence on the dimension $d$.

Another way to express the ``non-atomicity'' of the distributions of coordinates $X^j$, which is stable through linear combinations of independent variables, is the rate of decay of their Fourier transform.
Indeed, if $X^j$ is atomic, then its characteristic function does not vanish at infinity.
Proposition~\ref{prop:uniform-small-ball} below (proved in Section~\ref{sec:proof-proposition-independent}), which follows from an inequality by Esséen, provides uniform anti-concentration for one-dimensional marginals $\langle \theta, X\rangle$ in terms of the Fourier transform of the $X^j$, establishing Assumption~\ref{ass:small-ball} beyond bounded densities.
We let $\Phi_Z$ be the characteristic function of a real random variable $Z$, defined by $\Phi_Z (\xi) = \E [ e^{i \xi Z} ]$ for $\xi \in \R$.

\begin{proposition}
  \label{prop:uniform-small-ball}
  Assume that $X^1, \dots, X^d$ are independent and that there are constants $C_0 > 0$ and $\alpha \in (0, 1)$ such that, for $1 \leq j \leq d$ and $\xi \in \R$,
  \begin{equation}
    \label{eq:bound-characteristic}
    | \Phi_{X^j} (\xi) |
    \leq (1 + |\xi|/C_0)^{-\alpha}
    \, .
  \end{equation}
  Then, $X = (X^1, \dots, X^d)$ satisfies Assumption~\ref{ass:small-ball} with
  $C = 2^{1/\alpha} (2 \pi)^{1/\alpha - 1} (1 - \alpha)^{-1/\alpha} C_0$.
\end{proposition}

\section{Proof of Theorem~\ref{thm:lower-tail-covariance}}
\label{sec:proof-lower-tail}

\subsection{Truncation and small-ball condition}
\label{sec:trunc-small-ball}

The first step of the proof is to replace $X$ by the truncated vector $X' := \big( 1 \wedge \frac{\sqrt{d}}{\| X \|} \big) X$;
likewise, let $X_i' = \big( 1 \wedge \frac{\sqrt{d}}{\| X_i \|} \big) X_i$ for $1 \leq i \leq n$, and $\wh \Sigma_n' := n^{-1} \sum_{i=1}^n X_i' (X_i')^\top$.
Note that $X' (X')^\top %
\mleq X X^\top$ and $\| X' \| = \sqrt{d} \wedge \| X \|$, so that $\wh \Sigma_n' \mleq \wh \Sigma_n$ and $\E [ \| X' \|^2 ] \leq \E [ \| X \|^2] = d$.
It follows that $\lamin (\wh \Sigma_n') \leq \lamin (\wh \Sigma_n)$, hence it suffices to establish a lower bound for $\lamin (\wh \Sigma_n')$.

In addition, for every $\theta \in S^{d-1}$, $t \in (0, C^{-1})$ and $a \geq 1$,
\begin{align}  
  \P ( | \langle X', \theta \rangle | \leq t )
    &\leq \P \left( | \langle X, \theta \rangle | \leq a t \right) + \P \bigg( \frac{\sqrt{d}}{\| X \|} \leq \frac{1}{a} \bigg) \nonumber \\
    &\leq (C a t)^\alpha + \P ( \| X \| \geq a \sqrt{d}) \nonumber \\
    &\leq (C a t)^\alpha + \frac{\E [ \| X \|^2]}{a^2 d} \label{eq:truncation-small-ball-markov} \\
    &= (C t)^\alpha a^\alpha + \frac{1}{a^2}
      \label{eq:truncation-small-ball-1}
\end{align}
where we applied Markov's inequality in~\eqref{eq:truncation-small-ball-markov}.
In particular, letting $a = (C t)^{- \alpha / (2+\alpha)} $, inequality~\eqref{eq:truncation-small-ball-1} becomes
\begin{equation}
  \label{eq:truncation-small-ball-final}
  \P ( | \langle X', \theta \rangle | \leq t )
  \leq 2 (C t)^{2 \alpha / (2 + \alpha)}
  \, .
\end{equation}

\subsection{Concentration and PAC-Bayesian inequalities}
\label{sec:pac-bayes-ineq}

The smallest eigenvalue $\lamin (\wh \Sigma_n')$ of $\wh \Sigma_n'$ may be written as the infimum of an empirical process indexed by the unit sphere $S^{d-1} = \{ v \in \R^d : \| v \| = 1 \}$:
\begin{equation*}
  \lamin (\wh \Sigma_n')
  = \inf_{v \in S^{d-1}} \langle \wh \Sigma_n' v, v \rangle
  = \inf_{v \in S^{d-1}} \frac{1}{n} \sum_{i=1}^n \langle X_i', v\rangle^2
  \, .
\end{equation*}
Now, recall that the variables $\langle X_i', \theta \rangle^2$ are \iid and distributed as $\langle X', \theta \rangle^2$ for every $\theta \in S^{d-1}$.
The inequality~\eqref{eq:truncation-small-ball-final} on the left tail of this variable 
can be expressed in terms of its Laplace transform, through the following lemma:

\begin{lemma}
  \label{lem:small-ball-laplace}
  Let $Z$ be a nonnegative random variable.
  Assume that there exists $\alpha \in (0, 1 ]$ and $C > 0$ such that, for every $t \geq 0$,
  $\P (Z \leq t) \leq (C t)^{\alpha}$.
  Then, for every $\lambda > 0$,
  \begin{equation}
    \label{eq:small-ball-laplace}
    \E [ \exp (- \lambda Z) ]
    \leq (C / \lambda)^\alpha
    \, .
  \end{equation}
\end{lemma}

\begin{proof}[Proof of Lemma~\ref{lem:small-ball-laplace}]
  Since $0 \leq \exp (- \lambda Z) \leq 1$, we have
  \begin{align*}
    \E [ e^{- \lambda Z} ]
    = \int_0^1 \P (e^{- \lambda Z} \geq t) \di t %
    = \int_0^1 \P \bigg( Z \leq \frac{\log (1/t)}{\lambda} \bigg) \di t 
    \leq \int_0^1 \bigg( C \frac{\log (1/t)}{\lambda} \bigg)^\alpha \di t
    .
  \end{align*}
  Now, for $u > 0$, the map $\alpha \mapsto u^\alpha = e^{\alpha \log u}$ is convex on $\R$, so that $u^\alpha \leq \alpha u + (1 - \alpha)$ for $0 \leq \alpha \leq 1$.
  It follows that
  \begin{equation*}
    \int_0^1 \log^\alpha (1/t) \di t
    \leq \alpha \int_0^1 (- \log t) \di t + (1 - \alpha)
    = \alpha \big[ - t \log t + t \big]_0^1 + (1 - \alpha)
    = 1 ,
  \end{equation*}
  which establishes inequality~\eqref{eq:small-ball-laplace}.
\end{proof}

Here, inequality~\eqref{eq:truncation-small-ball-final} implies that, for every $\theta \in S^{d-1}$,
\begin{equation*}
  \P (\langle X', \theta \rangle^2 \leq t)
  = \P ( | \langle X', \theta \rangle | \leq \sqrt{t})
  \leq 2 (C \sqrt{t})^{2 \alpha / (2+\alpha)}
  = 2 (C^2 t)^{\alpha / (2 + \alpha)}
  \, .
\end{equation*}
Hence, Lemma~\ref{lem:small-ball-laplace} with $Z = \langle X', \theta \rangle^2$ implies that, for every $\lambda > 0$,
\begin{equation*}
  \E [ \exp ( - \lambda \langle X', \theta\rangle^2) ]
  \leq 2 (C^2 / \lambda)^{\alpha / (2+\alpha)}
  \, .
\end{equation*}
In other words, for $i = 1, \dots, n$, $\E [ \exp (Z_i (\theta)) ] \leq 1$, where, letting $\alpha' = \alpha / (2+\alpha)$, we define
\begin{equation*}
  Z_i (\theta) = - \lambda \langle X_i', \theta \rangle^2 + \alpha' \log \left( \frac{\lambda}{C^2} \right) - \log 2
\end{equation*}
with $\lambda > 0$ a fixed parameter that will be optimized later.
In particular, letting
\begin{equation*}
  Z (\theta) = Z_1 (\theta) + \dots + Z_n (\theta)
  = n \left[ - \lambda \langle \wh \Sigma_n' \theta, \theta \rangle + \alpha' \log \left( \frac{\lambda}{C^2} \right) - \log 2 \right]
  \, ,
\end{equation*}
the independence of $Z_1 (\theta), \dots, Z_n (\theta)$ implies that, for every $\theta \in S^{d-1}$,
\begin{equation}
  \label{eq:laplace-Z-sum}
  \E [ \exp ( Z (\theta) ) ]
  = \E [ \exp ( Z_1 (\theta) ) ] \cdots \E [ \exp ( Z_n (\theta) ) ]
  \leq 1
  \, .
\end{equation}
The bound~\eqref{eq:laplace-Z-sum} controls the upper tail of $Z (\theta)$ for fixed $\theta \in \Theta$.
In order to obtain a uniform control over $\theta$, similarly to~\cite{audibert2011robust,oliveira2016covariance} we will use the PAC-Bayesian technique for bounding empirical processes~\cite{mcallester1999pacbayesml,mcallester2003pac,catoni2007pacbayes}.
For completeness, we include a proof of Lemma~\ref{lem:pac-bayes}
(which is standard) below.

\begin{lemma}[PAC-Bayesian deviation bound]
  \label{lem:pac-bayes}
  Let $\Theta$ be a measurable space, and $Z (\theta)$, $\theta \in \Theta$, be a real-valued measurable process.
  Assume that $\E [\exp Z (\theta)] \leq 1$ for every $\theta \in \Theta$.
  Let $\pi$ be a probability distribution on $\Theta$.
  Then,
  \begin{equation}
    \label{eq:pac-bayes}
    \P \left( \forall \rho,
      \int_\Theta Z (\theta) \rho (\di \theta)
      \leq %
      \kll{\rho}{\pi} + t
    \right)
    \geq 1 - e^{-t}
    \, ,
  \end{equation}
  where $\rho$ spans all probability measures on $\Theta$, and
  $\kll{\rho}{\pi} := \int_\Theta \log \big( \frac{\di \rho}{\di \pi} \big) \di \rho \in [0, +\infty]$ is the Kullback-Leibler divergence between $\rho$ and $\pi$, and where we define the integral in~\eqref{eq:pac-bayes} to be $- \infty$ when the negative part is not integrable.
\end{lemma}

\begin{proof}[Proof of Lemma~\ref{lem:pac-bayes}]
  By integrating the inequality $\E [ \exp Z (\theta) ] \leq 1$ with respect to $\pi$ and using the Fubini-Tonelli theorem, we obtain
  \begin{equation}
    \label{eq:proof-pac-bayes-1}
    \E \bigg[ \int_\Theta \exp Z (\theta) \pi (\di \theta) \bigg]
    \leq 1
    \, .
  \end{equation}
  In addition, using the duality between the $\log$-Laplace transform and the Kullback-Leibler divergence %
  (see, e.g.,~\cite[p.~159]{catoni2004statistical}):
  \begin{equation*}
    \log \int_\Theta \exp (Z (\theta)) \pi (\di \theta)
    = \sup_\rho \bigg\{ \int_\Theta Z (\theta) \rho (\di \theta) - \kll{\rho}{\pi} \bigg\}
  \end{equation*}
  where the supremum spans over all probability distributions $\rho$ over $\Theta$, the inequality~\eqref{eq:proof-pac-bayes-1} writes
  \begin{equation}
    \label{eq:proof-pac-bayes-2}
    \E \left[ \exp \sup_{\rho} \left\{ \int_\Theta Z (\theta) \rho (\di \theta) - \kll{\rho}{\pi} \right\} \right]
    \leq 1 \, .
  \end{equation}
  Applying Markov's inequality to~\eqref{eq:proof-pac-bayes-2} yields the desired bound~\eqref{eq:pac-bayes}.
\end{proof}

Here, we let $\Theta = S^{d-1}$ and $Z (\theta)$ as defined above.
In addition, we take $\pi$ to be the uniform distribution on $S^{d-1}$, and for $v \in S^{d-1}$ and $\gamma > 0$ we define $\Theta (v, \gamma) := \{ \theta \in S^{d-1} : \| \theta - v \| \leq \gamma \}$ and let $\pi_{v, \gamma} = \pi (\Theta (v, \gamma))^{-1} \indic{\Theta (v, \gamma)} \cdot \pi$ be the uniform distribution over $\Theta (v, \gamma)$.
In this case, the PAC-Bayesian bound of Lemma~\ref{lem:pac-bayes} writes:
for every $t > 0$, with probability at least $1-e^{-t}$, for every $v \in S^{d-1}$ and $\gamma > 0$, 
\begin{equation}
  \label{eq:pac-bayes-small-ball}
  n \left[ - \lambda F_{v, \gamma} (\wh \Sigma_n') + \alpha' \log \left( \frac{\lambda}{C^2} \right) - \log 2 \right]
  \leq \kll{\pi_{v, \gamma}}{\pi} + t
  \, ,
\end{equation}
where we define for every symmetric matrix $\Sigma$:
\begin{equation}
  \label{eq:def-average-quadratic-posterior}
  F_{v, \gamma} (\Sigma)
  := \int_{\Theta} \langle \Sigma \theta, \theta \rangle \pi_{v, \gamma} (\di \theta)
  \, .
\end{equation}

\subsection{Control of the approximation term}
\label{sec:contr-appr-term}

Now, using the symmetries of the smoothing distributions $\pi_{v, \gamma}$, we will show that, for every $\gamma > 0$, $v \in S^{d-1}$ and symmetric matrix $\Sigma$, 
\begin{equation}
  \label{eq:approximation-term-expression}
  F_{v, \gamma} (\Sigma)
  = \big( 1 - \phi (\gamma) \big) \langle \Sigma v, v \rangle
  + \phi (\gamma) \cdot \frac{1}{d} \tr (\Sigma)
  \, ,
\end{equation}
where for $\gamma > 0$,
\begin{equation}
  \label{eq:approximation-term-phi}
  \phi (\gamma)
  := \frac{d}{d-1} \int_{\Theta} \big( 1 - \langle \theta, v\rangle^2 \big) \pi_{v, \gamma} (\di \theta) \in [ 0, d / (d-1) \gamma^2 ]
  \, .
\end{equation}

First, note that
\begin{equation*}
  F_{v, \gamma} (\Sigma)
  = \tr (\Sigma A_{v, \gamma}) \, ,
  \qquad \mbox{where}
  \qquad
  A_{v, \gamma} := \int_{\Theta} \theta \theta^\top \pi_{v, \gamma} (\di \theta)
  \, .
\end{equation*}
In addition, for every isometry $U \in O (d)$ of $\R^d$ and $v \in S^{d-1}$, $\gamma > 0$, the image measure $U_* \pi_{v, \gamma}$ of $\pi_{v, \gamma}$ under $U$ is $\pi_{U v, \gamma}$ (since $U$ sends $\Theta(v, \gamma)$ to $\Theta (Uv, \gamma)$ and preserves the uniform distribution $\pi$ on $S^{d-1}$).
It follows that
\begin{equation}
  \label{eq:approx-conjug-rotation}
  U A_{v, \gamma} U^{-1}
  = \int_{\Theta} (U \theta) (U \theta)^\top \pi_{v, \gamma} (\di \theta)
  = \int_{\Theta} \theta \theta^\top \pi_{U v, \gamma} (\di \theta)
  = A_{U v, \gamma}
  \, .
\end{equation}
In particular, $A_{v, \gamma}$ commutes with every isometry $U \in O (d)$ such that $U v = v$.
Taking $U$ to be the orthogonal reflection with respect to $H_v := (\R v)^\perp$, $A_{v, \gamma}$ preserves $\ker (U - I_d) = \R v$ and is therefore of the form $\phi_1 (v, \gamma) v v^\top + C_{v, \gamma}$ where $\phi_1 (v, \gamma) \in \R$ and $C_{v, \gamma}$ is a symmetric operator with $C_{v, \gamma} H_v \subset H_v$ and $C_{v, \gamma} v = v$.
Next, taking $U = v v^\top + U_v$ where $U_v$ is an arbitrary isometry of $H_v$, it follows that $C_{v, \gamma}$ commutes on $H_v$ with all isometries $U_v$, and is therefore of the form $\phi_2 (v, \gamma) P_v$, where $P_v = I_d - v v^\top$ is the orthogonal projection on $H_v$ and $\phi_2 (v, \gamma) \in \R$.
To summarize, we have:
\begin{equation*}
  A_{v, \gamma} = \phi_1 (v, \gamma) v v^\top + \phi_2 (v, \gamma) (I_d - v v^\top)
  \, .
\end{equation*}
Now, the identity~\eqref{eq:approx-conjug-rotation} shows that, for every $U \in O (d)$ and $v, \gamma$, $\phi_1 (U v, \gamma) = \phi_1 (v, \gamma)$ and $\phi_2 (U v, \gamma) = \phi_2 (v, \gamma)$; hence, these constants do not depend on $v$ and are simply denoted $\phi_1 (\gamma), \phi_2 (\gamma)$.
Defining $\phi (\gamma) := d \cdot \phi_2 (\gamma)$ and $\wt \phi (\gamma) := \phi_1 (\gamma)  - \phi_2 (\gamma)$, we therefore have:
\begin{equation}
  \label{eq:approx-A-decomp-phi}
  A_{v, \gamma} = \wt \phi (\gamma) v v^\top + \phi (\gamma) \cdot \frac{1}{d} I_d
  \, .
\end{equation}

Next, observe that
\begin{equation}
  \label{eq:approx-mixture-rotation}
  \int_{S^{d-1}} \pi_{v, \gamma} \pi (\di v)
  = \pi \, ;
\end{equation}
this follows from the fact that the measure $\pi'$ on the left-hand side of~\eqref{eq:approx-mixture-rotation} is a probability distribution on $S^{d-1}$ invariant under any $U \in O (d)$, since
\begin{align*}
  U_* \pi'
  &= \int_{S^{d-1}} U_* \pi_{v, \gamma} \pi (\di v)
  = \int_{S^{d-1}} \pi_{Uv, \gamma} \pi (\di v)
  = \int_{S^{d-1}} \pi_{v, \gamma} \pi (\di v)
  = \pi'
  \, .
\end{align*}
Equation~\eqref{eq:approx-mixture-rotation}, together with Fubini's theorem, implies that
\begin{equation*}
  \int_{S^{d-1}} A_{v, \gamma} \pi (\di v)
  = \int_{S^{d-1}} %
  \int_{S^{d-1}} \theta \theta^\top \pi_{v, \gamma} (\di \theta) %
  \pi (\di v)
  = \int_{S^{d-1}} \theta \theta^\top \pi (\di \theta)
  =: A
  \, .
\end{equation*}
Since $A$ commutes with isometries (by invariance of $\pi$), it is of the form $c I_d$ with $c = \tr (A) / d = (1/d) \int_{S^{d-1}} \| \theta \|^2 \pi (\di \theta) = 1/d$.
Plugging~\eqref{eq:approx-A-decomp-phi} into the previous equality, we obtain
\begin{equation*}
  \frac{1}{d} I_d
  = \int_{S^{d-1}} \Big[ \wt \phi (\gamma) v v^\top + \phi (\gamma) \cdot \frac{1}{d} I_d \Big] \pi (\di v)
  = \frac{1}{d} \wt \phi (\gamma) I_d + \frac{1}{d} \phi (\gamma) I_d
  \, ,
\end{equation*}
so that $\wt \phi (\gamma) = 1 - \phi (\gamma)$.
The decomposition~\eqref{eq:approx-A-decomp-phi} then writes:
\begin{equation*}
  A_{v, \gamma}
  = \big( 1 - \phi (\gamma) \big) v v^\top + \phi (\gamma) \cdot \frac{1}{d} I_d
  \, .
\end{equation*}
Recalling that $F_{v, \gamma} (\Sigma) = \tr (\Sigma A_{v, \gamma})$, we obtain the desired expression~\eqref{eq:approximation-term-expression} for $F_{v, \gamma}$.

Finally, note that on the one hand,
\begin{align*}
  \langle A_{v, \gamma} v, v \rangle
  = (1 - \phi (\gamma)) \| v \|^2 + \phi (\gamma) \cdot \frac{1}{d} \| v \|^2
  = 1 - \frac{d-1}{d} \phi (\gamma)
  \, ,
\end{align*}
while on the other hand:
\begin{align*}
  \langle A_{v, \gamma} v, v \rangle
  = \int_{S^{d-1}} \langle \theta, v \rangle^2 \pi_{v, \gamma} (\di \theta)
  \, ,
\end{align*}
so that
\begin{equation*}
  \phi (\gamma)
  = \frac{d}{d-1} \int_{S^{d-1}} \big( 1 - \langle \theta, v \rangle^2 \big) \pi_{v, \gamma} (\di \theta) \geq 0
  \, ,
\end{equation*}
where we used that $\langle \theta, v\rangle^2 \leq 1$ by the Cauchy-Schwarz inequality.

Now, let $\alpha$ denote the angle between $\theta$ and $v$.
We have $\langle \theta, v\rangle = \cos \alpha$ and $\| \theta - v \|^2 = (1 - \cos \alpha)^2 + \sin^2 \alpha = 2 (1 - \cos \alpha)$, so that $\langle \theta, v \rangle = 1 - \frac{1}{2} \| \theta - v \|^2$.
Since $\pi_{v, \gamma} (\di \theta)$-almost surely, $\| \theta - v \| \leq \gamma$, this implies
\begin{equation*}
  1 - \langle \theta, v\rangle^2
  = 1 - \Big( 1 - \frac{1}{2} \| \theta - v \|^2 \Big)^2
  = \| \theta - v \|^2 - \frac{1}{4} \| \theta - v \|^4
  \leq \gamma^2
  \, .
\end{equation*}
Integrating this inequality over $\pi_{v, \gamma}$ yields $\phi (\gamma) \leq d / (d-1) \gamma^2$; this establishes~\eqref{eq:approximation-term-phi}.

\subsection{Control of the entropy term}
\label{sec:control-entropy-term}

We now turn to the control of the entropy term in~\eqref{eq:pac-bayes-small-ball}.
Specifically, we will
show that, for every $v \in S^{d-1}$ and $\gamma > 0$, 
\begin{equation}
  \label{eq:entropy-term}
  \kll{\pi_{v, \gamma}}{\pi}
  \leq d \log \Big( 1 + \frac{2}{\gamma} \Big)
  \, .
\end{equation}

First, since %
$\di \pi_{v, \gamma} / \di \pi = \pi [ \Theta (v, \gamma) ]^{-1}$ $\pi_{v, \gamma}$-almost surely, $\kll{\pi_{v, \gamma}}{\pi} = \log {\pi [ \Theta (v, \gamma) ]}^{-1}$.
Now, let $N = N_c (\gamma, S^{d-1})$ denote the $\gamma$-covering number of $S^{d-1}$, namely the smallest $N \geq 1$ such that there exists $\theta_1, \dots, \theta_N \in S^{d-1}$ with
\begin{equation}
  \label{eq:covering-sphere}
  S^{d-1} = \bigcup_{i = 1}^N \Theta (\theta_i, \gamma)
  \, .
\end{equation}
Applying a union bound to~\eqref{eq:covering-sphere} and using the fact that $\pi [ \Theta (\theta_i, \gamma) ] = \pi [\Theta (v, \gamma)]$
yields $1 \leq N \pi [ \Theta (v, \gamma)]$, namely
\begin{equation}
  \label{eq:kl-covering}
  \kll{\pi_{v, \gamma}}{\pi}
  \leq \log N
  \, .
\end{equation}
Now, let $N_p (\gamma, S^{d-1})$ denote the $\gamma$-packing number of $S^{d-1}$, which is the largest number of points in $S^{d-1}$ with pairwise distances at least $\gamma$.
We have,
denoting $B^d = \{ x \in \R^d : \| x \| \leq 1 \}$,
\begin{equation}
  \label{eq:bound-covering-sphere}
  N \leq N_p (\gamma, S^{d-1})
  \leq N_p (\gamma, B^{d})
  \leq \left( 1 + \frac{2}{\gamma} \right)^d
  ,
\end{equation}
where the first inequality follows from a comparison of covering and packing numbers \cite[Lemma~4.2.8]{vershynin2018high}, the second one from the inclusion $S^{d-1} \subset B^d$ and the last one from a volumetric argument \cite[Lemma 4.2.13]{vershynin2018high}.
Combining~\eqref{eq:kl-covering} and~\eqref{eq:bound-covering-sphere} establishes~\eqref{eq:entropy-term}.

\subsection{Conclusion of the proof}
\label{sec:conclusion-proof}

First note that, since $\| X_i' \|^2 = \| X_i \|^2 \wedge d \leq d$ for $1 \leq i \leq n$,
\begin{equation}
  \label{eq:trace-empirical-norms}
  \tr (\wh \Sigma_n')
  = \frac{1}{n} \sum_{i=1}^n \| X_i' \|^2
  \leq d
  \, .
\end{equation}
Putting together the previous bounds~\eqref{eq:pac-bayes-small-ball}, \eqref{eq:approximation-term-expression}, \eqref{eq:entropy-term} and~\eqref{eq:trace-empirical-norms}, we get with probability $1 - e^{-n u}$, for every $v \in S^{d-1}$, $\gamma \in (0, 1/2]$,
\begin{align*}
  \alpha' \log \left( \frac{\lambda}{C^2} \right) - \log 2 - \frac{d}{n} &\log \left( 1 + \frac{2}{\gamma} \right) - u
  \leq \lambda F_{v, \gamma} (\wh \Sigma_n') \\
  &= \lambda \Big( (1 - \phi (\gamma)) \langle \wh \Sigma_n' v, v\rangle + \phi (\gamma) \cdot \frac{1}{d} \tr (\wh \Sigma_n') \Big) \\
  &\leq \lambda \left[ (1 - \phi (\gamma)) \langle \wh \Sigma_n' v, v\rangle + \phi (\gamma) \right]
\end{align*}
In particular, rearranging, and using the fact that $\phi (\gamma) \leq 1/2$ for $\gamma \leq 1/2$, as well as $\phi (\gamma) \leq \gamma^2$ and $\lamin (\wh \Sigma_n') = \inf_v \langle \wh \Sigma_n' v, v\rangle$, we get with probability $1 - e^{- n u}$,
\begin{equation}
  \label{eq:lower-bound-lambda-gamma}
  \lamin (\wh \Sigma_n')
  \geq \frac{2}{\lambda} \left[ \alpha' \log \left( \frac{\lambda}{C^2} \right) - \log 2 - \frac{d}{n} \log \left( 1 + \frac{2}{\gamma} \right) - u \right] - 2 \gamma^2 
\end{equation}

We first approximately maximize the above lower bound in $\gamma$, given $\lambda$.
Since $\gamma \leq 1/2$, $1 + 2 / \gamma \leq 1 + 1/\gamma^2 \leq 5 / (4\gamma^2)$.
We are therefore led to minimize
\begin{equation*}
  \frac{2d}{\lambda n} \log \left( \frac{5}{4\gamma^2} \right) + 2 \gamma^2
\end{equation*}
over $\gamma^2 \leq 1/4$.
Now, let $\gamma^2 = d / (2 \lambda n)$, which belongs to the prescribed range if
\begin{equation}
  \label{eq:constraint-lambda}
  \lambda \geq \frac{2 d}{n}
  \, .
\end{equation}
For this choice of $\gamma$, %
the lower bound~\eqref{eq:lower-bound-lambda-gamma} becomes
\begin{align*}
  \lamin &(\wh \Sigma_n')
  \geq \frac{2}{\lambda} \left[ \alpha' \log \left( \frac{\lambda}{C^2} \right) - \log 2 - \frac{d}{n} \log \left( \frac{5 \lambda n}{2 d} \right) - u \right] - \frac{d}{\lambda n} \\
  &= \frac{2}{\lambda} \left[ \left( \alpha' - \frac{d}{n} \right) \log \lambda - \alpha' \log C^2 - \left\{ \log 2 + \frac{d}{n} \log \left( \frac{5 n}{2 d}\right) + \frac{d}{2n} \right\} - u \right]
\end{align*}
Now, recall that by assumption, $d/n \leq \alpha/6 \leq 1/6$, so that (by monotonicity of $x \mapsto - x \log x$ on $(0, e^{-1}]$, replacing $d/n$ by $1/6$%
) the term inside braces is smaller than $c_0 = 1.3$.
In addition, assume that $\lambda \geq C^4$, so that $\log (\lambda / C^4) \geq 0$; in this case, condition~\eqref{eq:constraint-lambda} is automatically satisfied, since $2 d/n \leq 1/3 \leq C^4$.
Finally, since $\alpha' = \alpha/ (2+\alpha) \geq \alpha / 3$ and $d/n \leq \alpha/6$, $\alpha' \leq 2 (\alpha' - d/n)$ and $\alpha' - d/n \geq \alpha / 6$, so that
\begin{equation*}
  \left( \alpha' - \frac{d}{n} \right) \log \lambda - \alpha' \log C^2
  \geq \left( \alpha' - \frac{d}{n} \right) \log \left( \frac{\lambda}{C^4} \right)
  \geq \frac{\alpha}{6} \log \left( \frac{\lambda}{C^4} \right)
  ,
\end{equation*}
the previous inequalities implies that, for every $\lambda \geq C^4$ and $u > 0$, with probability at least $1 - e^{-n u}$,
\begin{equation*}
  \lamin (\wh \Sigma_n')
  \geq \frac{2}{\lambda} \left[ \frac{\alpha}{6} \log \left( \frac{\lambda}{C^4} \right) - c_0 - u \right]
  = \frac{\alpha}{3 C^4} \frac{\log \lambda' - 6 \alpha^{-1} (c_0 + u)}{\lambda'}
\end{equation*}
where $\lambda' = \lambda / C^4 \geq 1$.
A simple analysis shows that for $c \in \R$, the function $\lambda' \mapsto (\log \lambda' - c)/ \lambda'$ admits a maximum on $(0, + \infty)$ of $e^{-c - 1}$, reached at $\lambda' = e^{c+1}$.
Here $c = 6 \alpha^{-1} (c_0 + u) > 0$, so that $\lambda' > e > 1$.
Hence, for every $u > 0$, with probability at least $1 - e^{-nu}$,
\begin{equation}
  \label{eq:lower-bound-u}
  \lamin (\wh \Sigma_n')
  \geq \frac{\alpha}{3 C^4} \exp \left( - 1 - \frac{6 (c_0 + u)}{\alpha} \right)
  \geq C'^{-1} e^{- 6 u / \alpha}
  =: t \, ,
\end{equation}
where we let
  $C' := 3 C^4 e^{1 + 9/\alpha}$
  (using the fact that $6 c_0 \leq 8$ and %
  $1/\alpha \leq e^{1/\alpha}$).
  Inverting the bound~\eqref{eq:lower-bound-u}, we obtain that for every $t < C'^{-1}$,
  \begin{equation*}
    \P \big( \lamin (\wh \Sigma_n') \leq t \big)
    \leq (C' t)^{\alpha n / 6}
    \, .
  \end{equation*}
  Since $\lamin (\wh \Sigma_n) \geq \lamin (\wh \Sigma_n')$, and since the bound trivially holds for $t \geq C'^{-1}$, this concludes the proof.  

\section{Proofs from Section~\ref{sec:minimax-ls}}
\label{sec:proofs-least-squares}

In this section, we gather the remaining proofs of results from Section~\ref{sec:minimax-ls} on least squares regression, namely those of Proposition~\ref{prop:degenerate-minimax-risk}, Theorem~\ref{thm:ls-minimax}, Proposition~\ref{prop:upper-well-constant}, Theorem~\ref{thm:upper-well-spec} and Proposition~\ref{prop:upper-misspecified}.

\subsection{Preliminary: risk of Ridge and OLS estimators}
\label{sec:proof-least-sqaures-preliminaries}

We start with general expressions for the risk, which will be used several times in the proofs.
Here, we assume that $(X, Y)$ is as in Section~\ref{sec:minimax-ls}, namely $\E [Y^2] < + \infty$, $\E [ \| X \|^2 ] < + \infty$ and $\Sigma := \E [X X^\top]$ is invertible.
Letting $\eps := Y - \langle \beta^*, X\rangle$ denote the error, where $\beta^* := \Sigma^{-1} \E [Y X]$ is the risk minimizer, we let $m (X) := \E [\eps \cond X] = \E [Y \cond X] - \langle \beta^*, X\rangle$ denote the {misspecification} (or {approximation}) error of the linear model, and $\sigma^2 (X) := \Var (\eps \cond X) = \Var (Y \cond X)$ denote the conditional variance of the noise.

\begin{lemma}[Risk of the Ridge estimator]
  \label{lem:risk-ridge-general}
  Assume that $(X, Y)$ is of the previous form.
  Let $\lambda \geq 0$, and assume that either $\lambda > 0$ or that $P_X$ is non-degenerate and $n \geq d$.
  The risk of the Ridge estimator $\ridgeb$, defined by
  \begin{equation}
    \label{eq:def-ridge-new}
    \ridgeb := \argmin_{\beta \in \R^d} \bigg\{ \frac{1}{n} \sum_{i=1}^n ( Y_i - \langle \beta, X_i \rangle )^2 + \lambda \| \beta \|^2 \bigg\}
    = \big(\wh \Sigma_n + \lambda \Id \big)^{-1} \cdot \frac{1}{n} \sum_{i=1}^n Y_i X_i
    \, ,
  \end{equation}
  equals
  \begin{align}
    \label{eq:risk-ridge-general}
    \E \big[ \excessrisk (\wh \beta_{\lambda, n}) \big]
    &= \E \bigg[ \bigg\| \frac{1}{n} \sum_{i=1}^n m (X_i) X_i - \lambda \beta^* \bigg\|_{(\wh \Sigma_n + \lambda \id)^{-1} \Sigma (\wh \Sigma_n + \lambda \id)^{-1}}^2 \bigg] + \nonumber \\
    &\quad + \frac{1}{n^2} \E \bigg[ \sum_{i=1}^n \sigma^2 (X_i) \| X_i \|_{(\wh \Sigma_n + \lambda \id)^{-1} \Sigma (\wh \Sigma_n + \lambda \id)^{-1}}^2 \bigg]
      \, .
  \end{align}
\end{lemma}

\begin{proof}%
  Since $Y_i = \langle \beta^*, X_i\rangle + \eps_i$ for $i =1, \dots, n$, and since $\langle \beta^*, X_i \rangle X_i = X_i X_i^\top \beta^*$, we have
  \begin{equation}
    \label{eq:proof-ridge-1}
    \frac{1}{n} \sum_{i=1}^n Y_i X_i
    = \wh \Sigma_n \beta^* + \frac{1}{n} \sum_{i=1}^n \eps_i X_i
    \, .
  \end{equation}
  Hence, the excess risk of %
  $\wh \beta_{\lambda, n}$ (which is well-defined by the assumptions) is
  \begin{align}    
    \E \big[ \excessrisk (\wh \beta_{\lambda, n}) \big]
    &= \E \bigg[ \bigg\| (\wh \Sigma_n + \lambda \id)^{-1} \bigg( \wh \Sigma_n \beta^* + \frac{1}{n} \sum_{i=1}^n \eps_i X_i \bigg) - \beta^* \bigg\|_{\Sigma}^2 \bigg] \nonumber \\
    &= \E \bigg[ \bigg\| (\wh \Sigma_n + \lambda \id)^{-1} \cdot \frac{1}{n} \sum_{i=1}^n \eps_i X_i - \lambda (\wh \Sigma_n + \lambda \id)^{-1} \beta^* \bigg\|_{\Sigma}^2 \bigg] \nonumber \\
    &= \E \bigg[ \E \bigg[ \bigg\| \frac{1}{n} \sum_{i=1}^n \eps_i X_i - \lambda \beta^* \bigg\|_{(\wh \Sigma_n + \lambda \id)^{-1} \Sigma (\wh \Sigma_n + \lambda \id)^{-1}}^2 \Big| X_1, \dots, X_n \bigg] \bigg] \nonumber \\
    &= \E \bigg[ \bigg\| \frac{1}{n} \sum_{i=1}^n m (X_i) X_i - \lambda \beta^* \bigg\|_{(\wh \Sigma_n + \lambda \id)^{-1} \Sigma (\wh \Sigma_n + \lambda \id)^{-1}}^2 \bigg] + \nonumber \\
    &\quad + \frac{1}{n^2} \E \bigg[ \sum_{i=1}^n \sigma^2 (X_i) \| X_i \|_{(\wh \Sigma_n + \lambda \id)^{-1} \Sigma (\wh \Sigma_n + \lambda \id)^{-1}}^2 \bigg]
      \label{eq:proof-ridge-second-moment} %
  \end{align}
  where~\eqref{eq:proof-ridge-second-moment} is obtained by expanding and using the fact that, for
  $i \neq j$,
  \begin{align*}
    \E \big[ \eps_i \eps_j \cond X_1, \dots, X_n \big]
    &= m (X_i) m (X_j) \, , \\
    \E \big[ \eps_i^2 \cond X_1, \dots, X_n \big]
    &= m (X_i)^2 + \sigma^2 (X_i)
      \, . \qedhere
  \end{align*}
\end{proof}

In the special case where $\lambda = 0$,
the previous risk decomposition becomes:

\begin{lemma}[Risk of the OLS estimator]
  \label{lem:risk-ols-general}
  Assume that $P_X$ is non-degenerate and $n \geq d$.
  Then, %
  \begin{equation}
    \label{eq:risk-ols-general}
    \E \big[ \excessrisk (\lsb) \big]
    = \E \bigg[ \bigg\| \frac{1}{n} \sum_{i=1}^n m (X_i) \wt X_i \bigg\|_{\wt \Sigma_n^{-2}}^2 \bigg] + \frac{1}{n^2} \E \bigg[ \sum_{i=1}^n \sigma^2 (X_i) \| \wt X_i \|_{\wt \Sigma_n^{-2}}^2 \bigg]
    \, ,
  \end{equation}
  where we let $\wt X_i = \Sigma^{-1/2} X_i$ and $\wt \Sigma_n = \Sigma^{-1/2} \wh \Sigma_n \Sigma^{-1/2}$.
\end{lemma}

\begin{proof}
  This follows from Lemma~\ref{lem:risk-ridge-general} and the fact that, when $\lambda = 0$, for every $x \in \R^d$,
  \begin{equation*}
    \big\| x \big\|_{(\wh \Sigma_n + \lambda \id)^{-1} \Sigma (\wh \Sigma_n + \lambda \id)^{-1}}
    = \big\| \Sigma^{-1/2} x \big\|_{\Sigma^{1/2} \wh \Sigma_n^{-1} \Sigma \wh \Sigma_n^{-1} \Sigma^{1/2}}
    = \| \Sigma^{-1/2} x \|_{\wt \Sigma_n^{-2}}
    \, .
    \qedhere
  \end{equation*}
\end{proof}

\subsection{Proof of Theorem~\ref{thm:ls-minimax} and Proposition~\ref{prop:degenerate-minimax-risk}}
\label{sec:proof-minimax-risk}

\paragraph{Upper bound on the minimax risk.}

We start with an upper bound on the risk the least-squares estimator over the class $\wellclass(P_X, \sigma^2)$.
As in Theorem~\ref{thm:ls-minimax}, we assume that $n \geq d$ and that $P_X$ is non-degenerate.
Let $(X, Y) \sim P \in \wellclass (P_X, \sigma^2)$, so that $m (X) = 0$ and $\sigma^2 (X) \leq \sigma^2$.
It follows from Lemma~\ref{lem:risk-ols-general} that 
\begin{align*}
  \E \big[ \excessrisk (\lsb) \big]
  &\leq \frac{\sigma^2}{n^2} \E \bigg[ \sum_{i=1}^n \sigma^2 (X_i) \| \wt X_i \|_{\wt \Sigma_n^{-2}}^2 \bigg]
  = \frac{\sigma^2}{n^2} \E \bigg[ \tr \bigg( \wt \Sigma_n^{-2} \sum_{i=1}^n \wt X_i \wt X_i^\top \bigg) \bigg] \\
  &= \frac{\sigma^2}{n} \E \tr (\wt \Sigma_n^{-1})
  \, .
\end{align*}
Hence, the maximum risk of the OLS estimator $\lsb$ over the class $\wellclass (P_X, \sigma^2)$
(and thus the minimax risk over this class)
is at most $\sigma^2 \E [ \tr (\wt \Sigma_n^{-1})  ] / n $.

\paragraph{Lower bound on the minimax risk.}

We now provide a lower bound on the minimax risk over $\gaussclass (P_X, \sigma^2)$.
We will in fact establish the lower bound both in the setting of Theorem~\ref{thm:ls-minimax} (namely, $P_X$ is non-degenerate and $n \geq d$) and that of Proposition~\ref{prop:degenerate-minimax-risk} (the remaining cases).
In particular, we do not assume for now that $P_X$ is non-degenerate or that $n \geq d$.

For $\beta^* \in \R^d$, let $P_{\beta^*}$ denote the joint distribution of $(X, Y)$ where $X \sim P_X$ and $Y = \langle \beta^*, X\rangle + \eps$ with $\eps \sim \gaussdist (0, \sigma^2)$ independent of $X$.
Now, consider the decision problem with model $\gaussclass (P_X, \sigma^2) = \{ P_{\beta^*} : \beta^* \in \R^d \}$, decision space $\R^d$ and loss function $\loss (\beta^*, \beta) = \excessrisk_{P_{\beta^*}} (\beta) = \| \beta - \beta^* \|_\Sigma^2$.
Let $\risk (\beta^*, \wh \beta_n) = \E_{\beta^*} [ \loss (\beta^*, \wh \beta_n) ]$ denote the risk under $P_{\beta^*}$ of a decision rule $\wh \beta_n$ (that is, an estimator of $\beta^*$ using an \iid sample of size $n$ from $P_{\beta^*}$), namely its expected excess risk.
Consider the prior $\Pi_\lambda = \gaussdist (0, \sigma^2 / (\lambda n) \id)$ on $\gaussclass (P_X, \sigma^2)$. %
A standard computation (see, \eg, \cite{gelman2013bayesian})
shows that the posterior
$\Pi_{\lambda} (\cdot \cond (X_1, Y_1), \dots, (X_n, Y_n))$ is $\gaussdist (\wh \beta_{\lambda, n}, (\sigma^2/n) \cdot (\wh \Sigma_n + \lambda \id)^{-1} )$.
Since the loss function $\loss$ is quadratic, the Bayes estimator under $\Pi_\lambda$ is the expectation of the posterior, which is %
$\wh \beta_{\lambda, n}$.
Hence, using the comparison between minimax and Bayes risks:
\begin{equation}
  \label{eq:bayes-minimax-gauss}
  \inf_{\wh \beta_n} \sup_{P_{\beta^*} \in \gaussclass (P_X, \sigma^2)} \risk (\beta^*, \wh \beta_n)
  \geq \inf_{\wh \beta_n} \E_{\beta^* \sim \Pi_{\lambda}} \big[ \risk (\beta^*, \wh \beta_n) \big]
  = \E_{\beta^* \sim \Pi_{\lambda}} \big[ \risk (\beta^*, \wh \beta_{\lambda, n}) \big]
  \, ,
\end{equation}
where the infimum is over all estimators $\wh \beta_n$.
Note that
the left-hand side of~\eqref{eq:bayes-minimax-gauss} is simply the minimax
excess risk over $\gaussclass (P_X, \sigma^2)$.
On the other hand, applying Lemma~\ref{lem:risk-ridge-general} with $m (X) = 0$ and $\sigma^2 (X) = \sigma^2$ and noting that
\begin{align*}
  \E \bigg[ \sum_{i=1}^n %
  \| X_i \|_{(\wh \Sigma_n + \lambda \id)^{-1} \Sigma (\wh \Sigma_n + \lambda \id)^{-1}}^2 \bigg]
  &= %
    \E \bigg[ \tr \bigg\{ (\wh \Sigma_n + \lambda \id)^{-1} \Sigma (\wh \Sigma_n + \lambda \id)^{-1} \sum_{i=1}^n X_i X_i^\top \bigg\} \bigg] \\
  &= n \, \E \big[ \tr \big\{ (\wh \Sigma_n + \lambda \id)^{-1} \Sigma (\wh \Sigma_n + \lambda \id)^{-1} \wh \Sigma_n \big\} \big]
  \, ,
\end{align*}
we obtain
\begin{equation*}
  \risk (\beta^*, \wh \beta_{\lambda, n})
  = \lambda^2 \,\E \Big[ \| \beta^* \|_{(\wh \Sigma_n + \lambda \id)^{-1} \Sigma (\wh \Sigma_n + \lambda \id)^{-1}}^2 \Big] + \frac{\sigma^2}{n} \E \big[ \tr \big\{ (\wh \Sigma_n + \lambda \id)^{-1} \Sigma (\wh \Sigma_n + \lambda \id)^{-1} \wh \Sigma_n \big\} \big]
  \, .
\end{equation*}
This implies that
\begin{align}
  \label{eq:bayes-risk-gaussian}
  \nonumber
  \E_{\beta^* \sim \Pi_{\lambda}} \big[ \risk (\beta^*, \wh \beta_{\lambda, n}) \big]
  &= \E_{\beta^* \sim \Pi_\lambda} \Big[ \lambda^2 \,\E \Big[ \| \beta^* \|_{(\wh \Sigma_n + \lambda \id)^{-1} \Sigma (\wh \Sigma_n + \lambda \id)^{-1}}^2 \Big] \Big] + \\ &\quad + \frac{\sigma^2}{n} \E \big[ \tr \big\{ (\wh \Sigma_n + \lambda \id)^{-1} \Sigma (\wh \Sigma_n + \lambda \id)^{-1} \wh \Sigma_n \big\} \big]  
\end{align}
where $\E$ simply denotes the expectation with respect to $(X_1, \dots, X_n) \sim P_X^{n}$.
Now, by Fubini's theorem, and since $\E_{\beta^*\sim \Pi_\lambda} [ \beta^* (\beta^*)^\top ] = \sigma^2 / (\lambda n) \id$, we have
\begin{align}
  \label{eq:bayes-risk-fubini}
  &\E_{\beta^* \sim \Pi_\lambda} \Big[ \lambda^2 \,\E \Big[ \| \beta^* \|_{(\wh \Sigma_n + \lambda \id)^{-1} \Sigma (\wh \Sigma_n + \lambda \id)^{-1}}^2 \Big] \Big] \nonumber \\
  =& \, \lambda^2 \cdot \E \Big[ \E_{\beta^* \sim \Pi_\lambda} \Big[ \tr \big\{ (\wh \Sigma_n + \lambda \id)^{-1} \Sigma (\wh \Sigma_n + \lambda \id)^{-1} \beta^* (\beta^*)^\top \big\} \Big] \Big] \nonumber \\
  =& \, \frac{\sigma^2}{n} \E \big[ \tr \big\{ (\wh \Sigma_n + \lambda \id)^{-1} \Sigma (\wh \Sigma_n + \lambda \id)^{-1} \lambda \id \big\} \big]
     \, .
\end{align}
Plugging~\eqref{eq:bayes-risk-fubini} into~\eqref{eq:bayes-risk-gaussian} shows that the Bayes risk under $\Pi_\lambda$ equals
\begin{equation}
  \label{eq:bayes-risk-gaussian-2}
  \frac{\sigma^2}{n} \E \big[ \tr \big\{ (\wh \Sigma_n + \lambda \id)^{-1} \Sigma (\wh \Sigma_n + \lambda \id)^{-1} (\wh \Sigma_n + \lambda \id) \big\} \big]
  = \frac{\sigma^2}{n} \E \big[ \tr \big\{ (\wh \Sigma_n + \lambda \id)^{-1} \Sigma \big\} \big]
  \, .
\end{equation}
Hence, by~\eqref{eq:bayes-minimax-gauss} the minimax risk is larger than $({\sigma^2}/{n}) \cdot \E [ \tr \{ (\wh \Sigma_n + \lambda \id)^{-1} \Sigma \} ]$ for every $\lambda > 0$.
We now distinguish the settings of Theorem~\ref{thm:ls-minimax} and Proposition~\ref{prop:degenerate-minimax-risk}.

\emph{Degenerate case.}
First, assume that $P_X$ is degenerate or that $n < d$.
By Fact~\ref{prop:equiv-degenerate}, with probability $p > 0$, the matrix $\wh \Sigma_n$ is non-invertible.
When this occurs, let $\theta \in \R^d$ be such that $\| \theta \| = 1$ and $\wh \Sigma_n (\Sigma^{-1/2} \theta) = 0$.
We then have, for every $\lambda > 0$, 
\begin{equation*}
  \langle \Sigma^{-1/2} (\wh \Sigma_n + \lambda \id) \Sigma^{-1/2} \theta, \theta \rangle
  = 0 + \lambda \| \Sigma^{-1/2} \theta \|^2
  \leq \lambda \cdot \lamin^{-1} ,
\end{equation*}
where $\lamin = \lamin (\Sigma)$ denotes the smallest eigenvalue of $\Sigma$.
This implies that
\begin{align*}
  \tr \{ \Sigma^{1/2} (\wh \Sigma_n + \lambda \id)^{-1} \Sigma^{1/2} \}
  &\geq \lamax (\Sigma^{1/2} (\wh \Sigma_n + \lambda \id)^{-1} \Sigma^{1/2} ) \\
  &= \lamin^{-1} (\Sigma^{-1/2} (\wh \Sigma_n + \lambda \id) \Sigma^{-1/2})
  \geq \frac{\lamin}{\lambda}
\end{align*}
so that
\begin{equation}
  \label{eq:lower-bayes-risk-degenerate}
  \frac{\sigma^2}{n} \E \big[ \tr \big\{ (\wh \Sigma_n + \lambda \id)^{-1} \Sigma \big\} \big]
  \geq \frac{\sigma^2}{n} \cdot p \cdot \frac{\lamin}{\lambda}
  \, .
\end{equation}
Recalling that the left-hand side of equation~\eqref{eq:lower-bayes-risk-degenerate} is a lower bound on the minimax risk for every $\lambda > 0$, and noting that the right-hand side tends to $+ \infty$ as $\lambda \to 0$, %
the minimax risk is infinite as claimed in Proposition~\ref{prop:degenerate-minimax-risk}.

\emph{Non-degenerate case.}
Now, assume that $P_X$ is non-degenerate and that $n \geq d$.
By Fact~\ref{prop:equiv-degenerate}, $\wh \Sigma_n$ is invertible almost surely.
In addition, $\tr \{ (\wh \Sigma_n + \lambda \id)^{-1} \Sigma \} = \tr \{ (\Sigma^{-1/2} \wh \Sigma_n \Sigma^{-1/2} + \lambda \Sigma^{-1})^{-1} \}$ is decreasing in $\lambda$ (since $\lambda \mapsto \Sigma^{-1/2} \wh \Sigma_n \Sigma^{-1/2} + \lambda \Sigma^{-1}$ is increasing in $\lambda$), positive, and converges as $\lambda \to 0^+$ to $\tr (\wt \Sigma_n^{-1})$.
By the monotone convergence theorem, it follows that
\begin{equation}
  \label{eq:lower-bayes-risk-nondegenerate}
  \lim_{\lambda \to 0^+} \frac{\sigma^2}{n} \E \big[ \tr \big\{ (\wh \Sigma_n + \lambda \id)^{-1} \Sigma \big\} \big]
  = \frac{\sigma^2}{n}  \E \big[ \tr (\wt \Sigma_n^{-1}) \big]
  \, ,
\end{equation}
where the limit in the right-hand side belongs to $(0, + \infty]$.
Since the left-hand side is a lower bound on the minimax risk, the minimax risk over $\gaussclass (P_X, \sigma^2)$ is larger than $(\sigma^2/n) \E [ \tr (\wt \Sigma_n^{-1}) ]$.

\paragraph{Conclusion.}

Since $\gaussclass (P_X, \sigma^2) \subset \wellclass (P_X, \sigma^2)$, the minimax risk over $\wellclass (P_X, \sigma^2)$ is at least as large as that over $\gaussclass (P_X, \sigma^2)$.
When $P_X$ is degenerate or $n < d$, we showed that the minimax risk over $\gaussclass (P_X, \sigma^2)$ is infinite, establishing Proposition~\ref{prop:degenerate-minimax-risk}.
When $P_X$ is non-degenerate and $n \geq d$,
the minimax risk over $\wellclass (P_X, \sigma^2)$ is smaller, and the minimax risk over $\gaussclass (P_X, \sigma^2)$ larger, than $(\sigma^2/n) \E [ \tr (\wt \Sigma_n^{-1}) ]$,
so that these quantities agree and equal $(\sigma^2/n) \E [ \tr (\wt \Sigma_n^{-1}) ]$, as claimed in Theorem~\ref{thm:ls-minimax}.

\subsection{Proof of Theorem~\ref{thm:upper-well-spec}}
\label{sec:proof-trace-inverse}

The proof starts with the following lemma.

\begin{lemma}
  \label{lem:trace-inverse-approx}
  For any positive symmetric $d \times d$ matrix $A$ and $p \in [1, 2]$,
  \begin{equation}
    \label{eq:trace-inverse-approx}
    \tr (A^{-1}) + \tr (A) - 2 d
    \leq \max (1, \lamin(A)^{-1}) \cdot \tr \big( | A - I_d |^{2/p} \big)
    \, .
  \end{equation}  
\end{lemma}

\begin{proof}[Proof of Lemma~\ref{lem:trace-inverse-approx}]
  Let us start by showing that, for every $a > 0$,
  \begin{equation}
    \label{eq:scalar-inverse-approx}
    a^{-1} + a - 2
    \leq \max (1, a^{-1}) \cdot | a - 1 |^{2/p}
    \, .
  \end{equation}
  Multiplying both sides of~\eqref{eq:scalar-inverse-approx} by $a>0$, it amounts to
  \begin{equation*}
    (a-1)^2 = %
    1 + a^2 - 2 a
    \leq \max (a, 1) \cdot | a - 1 |^{2/p}
    \, ,
  \end{equation*}
  namely to $| a - 1 |^{2 - 2/p} \leq \max (a, 1)$.
  For $a \in (0, 2]$, this inequality holds since $|a-1| \leq 1$ and $2 - 2/p \geq 0$, so that $| a - 1|^{2 - 2/p} \leq 1 \leq \max (a, 1)$.
  For $a \geq 2$, the inequalities $| a - 1|\geq 2$ and $2 - 2/p \leq 1$ imply that $| a - 1|^{2 - 2/p} \leq |a - 1| \leq a \leq \max (a, 1)$.
  This establishes~\eqref{eq:scalar-inverse-approx}.

  Now, let $a_1, \dots, a_d > 0$ be the eigenvalues of $A$.
  Without loss of generality, assume that $a_d = \min_j (a_j) = \lamin (A)$.
  Then, by %
  inequality~\eqref{eq:scalar-inverse-approx}
  and the bound $\max(1, a_j^{-1}) \leq \max (1, a_d^{-1})$, we have
  \begin{equation*}
    \tr (A^{-1}) + \tr (A) - 2 d
    = \sum_{j=1}^d ( a_j^{-1} + a_j - 2 ) %
    \leq \max (1, a_d^{-1}) \sum_{j=1}^d |a_j - 1|^{2/p}
    \, ,
  \end{equation*}
  which is precisely the desired inequality~\eqref{eq:trace-inverse-approx}.
\end{proof}

\begin{proof}[Proof of Theorem~\ref{thm:upper-well-spec}]
  Let $p \in (1, 2]$ which will be determined later, and denote $q := p / (p-1)$ its complement.
Applying Lemma~\ref{lem:trace-inverse-approx} to $A = \wt \Sigma_n$ yields:
\begin{equation*}
  \tr (\wt \Sigma_n^{-1}) + \tr (\wt \Sigma_n) - 2 d
  \leq \max (1, \lamin(\wt \Sigma_n)^{-1}) \cdot \tr \big( | \wt \Sigma_n - I_d |^{2/p} \big)
  \, .
\end{equation*}
Since $\E [ \tr (\wt \Sigma_n)] = d$, taking the expectation in the above bound and dividing by $d$ yields:
\begin{align}
  \frac{1}{d} \E \big[ \tr (\wt \Sigma_n^{-1}) \big] - 1
  &\leq \E \Big[ \max (1, \lamin(\wt \Sigma_n)^{-1}) \cdot \frac{1}{d} \tr \big( | \wt \Sigma_n - I_d |^{2/p} \big) \Big] \nonumber \\
  &\leq \E \big[ \max (1, \lamin(\wt \Sigma_n)^{-1})^q \big]^{1/q} \cdot \E \Big[ \Big( \frac{1}{d} \tr \big( | \wt \Sigma_n - I_d |^{2/p} \big) \Big)^p \Big]^{1/p} \label{eq:proof-trace-inverse-holder} \\
  &\leq \E \big[ \max (1, \lamin(\wt \Sigma_n)^{-q} ) \big]^{1/q} \cdot \E \Big[ \frac{1}{d} \tr \big( ( \wt \Sigma_n - I_d )^{2} \big) \Big]^{1/p} \label{eq:proof-trace-inverse-jensen}
\end{align}
where~\eqref{eq:proof-trace-inverse-holder} comes from Hölder's inequality, while~\eqref{eq:proof-trace-inverse-jensen} is obtained by noting that $x \mapsto x^p$ is convex and that $(1/d) \tr (A)$ is the average of the eigenvalues of the symmetric matrix $A$.
Next, 
\begin{align}
  \E \Big[ \frac{1}{d} \tr \big( (\wt \Sigma_n - I_d)^2 \big) \Big]
  &= \frac{1}{d} \tr \bigg\{ \E \bigg[ \bigg( \frac{1}{n} \sum_{i=1}^n ( \wt X_i \wt X_i^\top - I_d) \bigg)^2 \bigg] \bigg\} \nonumber \\
  &= \frac{1}{n^2 d} \tr \bigg\{ \sum_{1 \leq i, j \leq n} \E \big[ (\wt X_i \wt X_i^\top - I_d) (\wt X_j \wt X_j^\top - I_d) \big] \bigg\} \nonumber \\
  &= \frac{1}{n d} \tr \left\{ \E \big[ (\wt X \wt X^\top - I_d)^2 \big] \right\} \, ,
    \label{eq:proof-trace-inverse-cancels} 
\end{align}
where we used in~\eqref{eq:proof-trace-inverse-cancels} the fact that, for $i \neq j$, $\E \big[ (\wt X_i \wt X_i^\top  - I_d) (\wt X_j \wt X_j^\top  - I_d) \big] = \E [ \wt X_i \wt X_i^\top - I_d ] \E [ \wt X_j \wt X_j^\top - I_d ] = 0$.
Now,
for $x \in \R^d$,
\begin{equation*}
  \tr \{ (x x^\top - I_d)^2 \}
  = \tr \{ \| x \|^2 x x^\top - 2 x x^\top + I_d \}
  = \| x \|^4 - 2 \| x \|^2 + d
  \, ,
\end{equation*}
so that~\eqref{eq:proof-trace-inverse-cancels} becomes, as $\E [ \| \wt X \|^2] = d$ and $\E [ \| \wt X \|^4 ] \leq \kappa d^2$ (Assumption~\ref{ass:fourth-moment-norm}),
\begin{equation}
  \label{eq:proof-trace-inverse-fourth-final}
  \E \Big[ \frac{1}{d} \tr \big( (\wt \Sigma_n - I_d)^2 \big) \Big]
  = \frac{1}{n d} \Big( \E \| \wt X \|^4
    - 2 \E \| \wt X\|^2 + d \Big) %
  = \frac{1}{n} \Big( \frac{1}{d} \E \| \wt X \|^4 - 1 \Big)
  \leq \frac{\kappa d}{n}
  .
\end{equation}
In addition, recall that $\wt X$ satisfies Assumption~\ref{ass:small-ball} and that $n \geq \max (6d/\alpha, 12/\alpha)$. 
Hence, letting $C' \geq 1$ be the constant in Theorem~\ref{thm:lower-tail-covariance}, we have by Corollary~\ref{cor:negative-moments-smallest}:
\begin{equation}
  \label{eq:proof-inverse-trace-lamin-final}
  \E \big[ \max (1, \lamin (\wt \Sigma_n)^{-q}) \big]
  \leq 2 C'^q
  \, .
\end{equation}

Finally, plugging the bounds~\eqref{eq:proof-trace-inverse-fourth-final} and~\eqref{eq:proof-inverse-trace-lamin-final} into~\eqref{eq:proof-trace-inverse-jensen} and
letting $q = \alpha' n/2$, so that $1/p = 1 - 1/q = 1 - 2/(\alpha' n)$, we obtain
\begin{equation}
  \label{eq:proof-trace-inverse-final-upto-cst}
  \frac{1}{d} \cdot \E \big[ \tr (\wt \Sigma_n^{-1}) \big] - 1
  \leq (2 C'^q)^{1/q} \cdot \Big( \frac{\kappa d}{n} \Big)^{1/p}
  \leq 2 C' \cdot \frac{\kappa d}{n} \cdot \Big( \frac{n}{\kappa d} \Big)^{2 / (\alpha' n)}
  \, .
\end{equation}
Now, since $\kappa = \E [ \| \wt X \|^4 ] / \E [ \| \wt X \|^2]^2 \geq 1$ and $d \geq 1$,
\begin{equation*}
  \Big( \frac{n}{\kappa d} \Big)^{2/(\alpha' n)}
  \leq n^{2/(\alpha' n)}
  = \exp \Big( \frac{2 \log n}{\alpha' n} \Big)
  \, .
\end{equation*}
An elementary analysis shows that the function $g : x \mapsto \log x / x$ is increasing on $(0, e ]$ and decreasing on $[e, + \infty)$.
Hence, if $x, y > 1$ satisfy $x \geq y \log y \geq e$, then
\begin{equation*}
  \frac{\log x}{x}
  \leq \frac{\log y + \log \log y}{y \log y}
  \leq \frac{1 + e^{-1}}{y}
\end{equation*}
where we used $\log \log y / \log y \leq g (e) = e^{-1} $.
Here by assumption $n \geq 12 \alpha^{-1} \log (12 \alpha^{-1}) = 2 \alpha'^{-1} \log (2 \alpha'^{-1})$, and thus $\log n / n \leq (1 + e^{-1})/ (2 /\alpha')$, so that
\begin{equation*}
  \Big( \frac{n}{\kappa d} \Big)^{2/(\alpha' n)}
  \leq \exp \Big( \frac{2}{\alpha'} \cdot \frac{1 + e^{-1}}{2/\alpha'} \Big)
  = \exp \big( 1 + e^{-1} \big)
  \leq 4
  \, .
\end{equation*}
Plugging this inequality into~\eqref{eq:proof-trace-inverse-final-upto-cst} yields the desired bound~\eqref{eq:expected-inverse-trace}.
Equation~\eqref{eq:upper-bound-minimax} then follows by Theorem~\ref{thm:ls-minimax}.
\end{proof}

\subsection{Proof of Proposition~\ref{prop:upper-misspecified}}
\label{sec:proof-upper-misspecified}

Recall that, by Lemma~\ref{lem:risk-ols-general}, we have
\begin{equation}
  \label{eq:proof-risk-ols-misspecified}
  \E \big[ \excessrisk (\lsb) \big]
  = \E \bigg[ \bigg\| \frac{1}{n} \sum_{i=1}^n m (X_i) \Sigma^{-1/2} X_i \bigg\|_{\wt \Sigma_n^{-2}}^2 \bigg] + \frac{1}{n^2} \E \bigg[ \sum_{i=1}^n \sigma^2 (X_i) \big\| \Sigma^{-1/2} X_i \big\|_{\wt \Sigma_n^{-2}}^2 \bigg]
  \, .
\end{equation}
Now, since $\wt \Sigma_n^{-2} \leq \lamin (\wt \Sigma_n)^{-2} I_d$, we have for every random variable $V_n$:
\begin{align}
  \E \big[ \| V_n \|^2_{\wt \Sigma_n^{-2}} \big]
  &\leq %
      \E \big[ \| V_n \|^2 \big] + \E \big[ \big\{ \lamin (\wt \Sigma_n)^{-2} - 1 \big\}_+ \cdot \| V_n \|^2 \big] \nonumber \\
  &\leq \E \big[ \| V_n \|^2 \big] + \E \big[ \{ \lamin (\wt \Sigma_n)^{-2} - 1 \}_+^2 \big]^{1/2} \cdot \E \big[ \| V_n \|^4 \big]^{1/2} \, ,
    \label{eq:proof-misspecified-cauchy-schwarz}
\end{align}
where~\eqref{eq:proof-misspecified-cauchy-schwarz} follows from the Cauchy-Schwarz inequality.
Letting $V_n = \sigma (X_i) \Sigma^{-1/2} X_i$, we obtain from~\eqref{eq:proof-misspecified-cauchy-schwarz}
  \begin{align}
    \label{eq:proof-misspecified-varterm}
    &\frac{1}{n^2} \E \bigg[ \sum_{i=1}^n \sigma^2 (X_i) \big\| \Sigma^{-1/2} X_i \big\|_{\wt \Sigma_n^{-2}}^2 \bigg] \nonumber \\
    &\leq \frac{1}{n} \E \big[ \sigma^2 (X) \| \Sigma^{-1/2} X \|^2 \big]
      + \frac{1}{n} \E \big[ \{ \lamin (\wt \Sigma_n)^{-2} - 1 \}_+^2 \big]^{1/2} \E \big[ \sigma^4 (X) \| \Sigma^{-1/2} X \|^4 \big]^{1/2}
      \, .
  \end{align}
  On the other hand, let $V_n = n^{-1} \sum_{i=1}^n m (X_i) \Sigma^{-1/2} X_i$; we have, since $\E [ m(X_i) X_i ] = \E [ \eps_i X_i ] = 0$,
  \begin{align}
    \E \big[ \| V_n \|^2 \big]
    &= \E \bigg[ \Big\| \frac{1}{n} \sum_{i=1}^n m (X_i) X_i \Big\|_{\Sigma^{-1}}^2 \bigg] \nonumber \\
    &= \frac{1}{n^2} \sum_{1 \leq i, j \leq n} \E \big[ \left\langle m(X_i) X_i, m (X_j) X_j \right\rangle_{\Sigma^{-1}} \big] \nonumber \\
    &= \frac{1}{n^2} \sum_{i=1}^n \E \big[ m (X_i)^2 \| \Sigma^{-1/2} X_i \|^2 \big] + \frac{1}{n^2} \sum_{i \neq j} \big\langle \E [ m(X_i) X_i ], \E [ m(X_j) X_j ] \big\rangle_{\Sigma^{-1}} \nonumber \\
    &= \frac{1}{n} \E \big[ m (X)^2 \| \Sigma^{-1/2} X \|^2 \big]
      \label{eq:proof-misspecified-approxterm-m2}
      \, .
  \end{align}
  In addition,
  \begin{equation*}
    \E \big[ \| V_n \|^4 \big]
    = \frac{1}{n^4} \sum_{1\leq i,j,k,l \leq n} \E \big[ \left\langle m(X_i) X_i, m (X_j) X_j \right\rangle_{\Sigma^{-1}} \left\langle m(X_k) X_k, m (X_l) X_l \right\rangle_{\Sigma^{-1}} \big]
    .
  \end{equation*}
  Now, by independence and since $\E [ m (X) X] = 0$, each term in the sum above where one index among $i,j,k,l$ is distinct from the others cancels.
  We therefore have
  \begin{align}    
    \E \big[ \big\| V_n \big\|^4 \big]
    &= \frac{1}{n^4} \sum_{i=1}^n \E \big[ \| m (X_i) X_i \|_{\Sigma^{-1}}^4 \big] + \frac{2}{n^4} \sum_{i<j} \E \big[ \| m (X_i) X_i \|_{\Sigma^{-1}}^2 \| m (X_j) X_j \|_{\Sigma^{-1}}^2 \big] +
      \nonumber \\
    &\quad + \frac{4}{n^4} \sum_{1 \leq i < j \leq n} \E \big[ \langle m(X_i) X_i, m (X_j) X_j \rangle_{\Sigma^{-1}}^2 \big] \nonumber \\
    &\leq \frac{1}{n^4} \sum_{i=1}^n \E \big[ \| m (X_i) X_i \|_{\Sigma^{-1}}^4 \big] + \frac{6}{n^4} \sum_{i<j} \E \big[ \| m (X_i) X_i \|_{\Sigma^{-1}}^2 \| m (X_j) X_j \|_{\Sigma^{-1}}^2 \big]
      \label{eq:proof-misspecified-approxterm-cs} \\ 
    &= \frac{1}{n^3} \cdot \E \big[ m (X)^4 \| \Sigma^{-1/2} X \|^4 \big] + \frac{6}{n^4} \cdot \frac{n (n-1)}{2} \cdot \E \big[ m (X)^2 \| X \|_{\Sigma^{-1}}^2 \big]^2 \nonumber \\
    &\leq \frac{1}{n^3} \cdot \E \big[ m (X)^4 \| \Sigma^{-1/2} X \|^4 \big] + \frac{3}{n^2} \cdot \E \big[ m (X)^2 \| \Sigma^{-1/2} X \|^2 \big]^2 \nonumber \\
    &\leq \frac{4}{n^2} \cdot \E \big[ m (X)^4 \| \Sigma^{-1/2} X \|^4 \big]
      \label{eq:proof-misspecified-approxterm-m4-fin}
  \end{align}
  where~\eqref{eq:proof-misspecified-approxterm-cs} and~\eqref{eq:proof-misspecified-approxterm-m4-fin} rely on the Cauchy-Schwarz inequality.
  Hence, it follows from~\eqref{eq:proof-misspecified-cauchy-schwarz},~\eqref{eq:proof-misspecified-approxterm-m2} and~\eqref{eq:proof-misspecified-approxterm-m4-fin} that
  \begin{align}
    &\E \bigg[ \bigg\| \frac{1}{n} \sum_{i=1}^n m (X_i) \Sigma^{-1/2} X_i \bigg\|_{\wt \Sigma_n^{-2}}^2 \bigg]
      \nonumber \\
    &\leq \frac{1}{n} \E \big[ m (X)^2 \| \Sigma^{-1/2} X \|^2 \big] +
      \E \big[ \{ \lamin (\wt \Sigma_n)^{-2} - 1 \}_+^2 \big]^{1/2} \cdot \Big( \frac{4}{n^2} \cdot \E \big[ m (X)^4 \| \Sigma^{-1/2} X \|^4 \big] \Big)^{1/2} \nonumber \\
    &\leq \frac{1}{n} \E \big[ m (X)^2 \| \Sigma^{-1/2} X \|^2 \big] +
      \frac{2}{n} \E \big[ \{ \lamin (\wt \Sigma_n)^{-2} - 1 \}_+^2 \big]^{1/2}  \E \big[ m (X)^4 \| \Sigma^{-1/2} X \|^4 \big]^{1/2}
      \, .
      \label{eq:proof-misspecified-approxterm-fin}
  \end{align}
  Plugging~\eqref{eq:proof-misspecified-varterm} and~\eqref{eq:proof-misspecified-approxterm-fin} into the decomposition~\eqref{eq:proof-risk-ols-misspecified} yields:
  \begin{align}
    \label{eq:eq:proof-misspecified-risk-dec}
    \E \big[ \excessrisk (\lsb) \big]
    &\leq \frac{1}{n} \E \big[ \big( m (X)^2 + \sigma^2 (X) \big) \| \Sigma^{-1/2} X \|^2 \big] \nonumber
    + \frac{1}{n} \E \big[ \{ \lamin (\wt \Sigma_n)^{-2} - 1 \}_+^2 \big]^{1/2} \times \\
    &\quad \times \Big( \E \big[ \sigma^4 (X) \| \Sigma^{-1/2} X \|^4 \big]^{1/2} + 2 \E \big[ m (X)^4 \| \Sigma^{-1/2} X \|^4 \big]^{1/2} \Big)
  \end{align}

\paragraph{Oliveira's lower tail bound.}

\cite{oliveira2016covariance} showed that,
under Assumption~\ref{ass:equiv-l4-l2}, 
we have
\begin{equation*}
  \P \big( \lamin (\wh \Sigma_n) \geq 1 - \eps \big) \geq 1 - \delta
\end{equation*}
provided that
\begin{equation*}
  n \geq \frac{81 \kappa ( d + 2 \log (2/\delta))}{\eps^2}
  \, .
\end{equation*}
This can be rewritten as:
\begin{equation}
  \label{eq:proof-bound-oliveira}
  \P \bigg( \lamin (\wh \Sigma_n) < 1 - 9 \kappa^{1/2} \sqrt{\frac{d + 2 \log (2/\delta)}{n}} \bigg)
  \leq \delta
  \, .
\end{equation}

\paragraph{Bound on the remaining term.}

Since the function $x \mapsto x^2$ is $2$-Lipschitz on $[0, 1]$, we have $(x^{-2} - 1)_+ = (1 - x^2)_+ / x^2 \leq 2 (1 - x)_+ / x^2$ for $x > 0$, so that by Cauchy-Schwarz,
\begin{align}
  \label{eq:proof-misspecified-lamin}
  \E \big[ \{ \lamin (\wh \Sigma_n)^{-2} - 1 \}_+^2 \big]^{1/2}
  &\leq \E \Big[ \frac{4 \{ 1 - \lamin (\wh \Sigma_n) \}_+^2}{\lamin (\wh \Sigma_n)^4} \Big]^{1/2} \nonumber \\
  &\leq 2 \E \big[ \{ 1 - \lamin (\wh \Sigma_n) \}_+^4 \big]^{1/4} \E \big[ \lamin (\wh \Sigma_n)^{-8} \big]^{1/4}
    \, .
\end{align}
First, note that
\begin{align}
  \E \big[ \{ 1 - \lamin (\wh \Sigma_n) \}_+^4 \big]
  &= \int_{0}^\infty \P \big( \{ 1 - \lamin (\wh \Sigma_n) \}_+^4 \geq u \big) \di u \nonumber \\
  &= \int_{0}^1 \P \big( \lamin (\wh \Sigma_n) \leq 1 - u^{1/4} \big) \di u \nonumber \\
  &= \int_{0}^1 \P \big( \lamin (\wh \Sigma_n) \leq 1 - v^{1/2} \big) 2 v \di v
    \, .
    \label{eq:proof-misspecified-lamin-2}
\end{align}
Now, let $v^{1/2} = 9 \kappa^{1/2} \sqrt{[ d + 2 \log (2/\delta) ] / n}$, so that the bound~\eqref{eq:proof-bound-oliveira}
yields $\P (\lamin (\wh \Sigma_n) \leq 1 - v^{1/2}) \leq \delta$.
We have, equivalently,
\begin{equation*}
  \delta
  = 2 \exp \Big( - \frac{n}{162 \kappa} \Big( v - \frac{81 \kappa d}{n} \Big) \Big)
  \leq 2 \exp \Big( - \frac{n}{324 \kappa} v \Big)
\end{equation*}
as long as $v \geq 162 \kappa d / n$.
Plugging this inequality into~\eqref{eq:proof-misspecified-lamin-2} yields
\begin{align*}
  \E \big[ \{ 1 - \lamin (\wh \Sigma_n) \}_+^4 \big]
  &\leq \int_0^{\min({162 \kappa d}/{n}, 1)} 2 v \di v + \int_{\min(162 \kappa d / n, 1)}^1 2 \exp \Big( - \frac{n}{324 \kappa} v \Big) 2 v \di v \\
  &\leq \Big( \frac{162 \kappa d}{n} \Big)^2 + \Big( \frac{324 \kappa}{n} \Big)^2 \int_0^\infty 4 \exp (- w) w \di w \\
  &= \Big( \frac{162 \kappa d}{n} \Big)^2 + 4 \Big( \frac{324 \kappa}{n} \Big)^2
\end{align*}
so that, using the inequality $(x+y)^{1/4} \leq x^{1/4} + y^{1/4}$,
\begin{equation}
  \label{eq:proof-misspecified-lamin-3}
  \E \big[ \{ 1 - \lamin (\wh \Sigma_n) \}_+^4 \big]^{1/4}
  \leq 9 \sqrt{\frac{2 \kappa d}{n}} + 18 \sqrt{\frac{2 \kappa}{n}}
  \leq 27 \sqrt{\frac{2 \kappa d}{n}}
  \, .
\end{equation}
Also, by Corollary~\ref{cor:negative-moments-smallest} %
and the fact that $\alpha n /12 \geq 8$, 
$\E [ \lamin (\wh \Sigma_n)^{-8} ] \leq 2 C'^8$, so that inequality~\eqref{eq:proof-misspecified-lamin} becomes
\begin{equation}
  \label{eq:proof-misspecified-lamin-4}
  \E \big[ \{ \lamin (\wh \Sigma_n)^{-2} - 1 \}_+^2 \big]^{1/2}
  \leq 2 \times 27 \sqrt{\frac{2 \kappa d}{n}} \times 2^{1/4} C'^2
  \leq 92 C'^2 \sqrt{\frac{\kappa d}{n}}
  \, .
\end{equation}

\paragraph{Final bound.}

Now, let $\chi > 0$ as in Proposition~\ref{prop:upper-misspecified}.
Since
\begin{equation*}
  \E [ \eps^2 \cond X ]
  = m (X)^2 + \sigma^2 (X)
  \geq \max (m (X)^2, \sigma^2 (X))
  \, ,
\end{equation*}
we have
\begin{align}
  \label{eq:kurtosis-noise}
  &\max \Big( \E \big[ m (X)^4 \| \Sigma^{-1/2} X \|^4 \big], \E \big[ \sigma^4 (X) \| \Sigma^{-1/2} X \|^4 \big] \Big) \nonumber \\
  &\leq \E [ \E [ \eps^2 \cond X ]^2 \| \Sigma^{-1/2} X \|^4 ]
  = {\chi d^2}
  \, .
\end{align}
Putting the bounds~\eqref{eq:proof-misspecified-lamin-4} and~\eqref{eq:kurtosis-noise} inside~\eqref{eq:eq:proof-misspecified-risk-dec} yields
\begin{align}
  \label{eq:proof-misspecified-gen}
  \E \big[ \excessrisk (\lsb) \big]
  &\leq \frac{1}{n} \E \big[ \big( m (X)^2 + \sigma^2 (X) \big) \| \Sigma^{-1/2} X \|^2 \big]
  + \frac{1}{n} \cdot 92 C'^2 \sqrt{\frac{\kappa d}{n}} \cdot 3 \sqrt{\chi} d \nonumber \\
  &= \frac{1}{n} \E \big[ (Y - \langle \beta^*, X\rangle)^2 \| \Sigma^{-1/2} X \|^2 \big] + 276 C'^2 \sqrt{\kappa \chi} \Big( \frac{d}{n} \Big)^{3/2}
    \, ,
\end{align}
where we used the fact that $\E [ (Y - \langle \beta^*, X\rangle)^2 \cond X ] = m (X)^2 + \sigma^2 (X) $.
This establishes~\eqref{eq:upper-misspecified}.
Finally, if $P \in \misclass (P_X, \sigma^2)$, then
$\E [ \eps^2 \cond X ] \leq \sigma^2$, so that
\begin{equation*}
  \chi 
  = \E [ \E [ \eps^2 \cond X]^2 \| \Sigma^{-1/2} X \|^4 ] / d^2
  \leq \sigma^4 \E [ \| \Sigma^{-1/2} X \|^4 ] / d^2
  \leq \sigma^4 \kappa
  \, ,
\end{equation*}
where we used the fact that $\E [ \| \Sigma^{-1/2} X \|^4 ] \leq \kappa d^2$ by Assumption~\ref{ass:equiv-l4-l2} (see Remark~\ref{rem:fourth-moment-assumptions}).
Plugging this inequality, together with $\E [ (Y - \langle \beta^*, X\rangle)^2 \| \Sigma^{-1/2} X \|^2 ] \leq \sigma^2 d$, inside~\eqref{eq:proof-misspecified-gen}, yields the upper bound~\eqref{eq:upper-misspecified-minimax}.
This concludes the proof.

\section{Remaining proofs from Section~\ref{sec:small-ball-bounds}}
\label{sec:remaining-proofs}

In this section, we gather the proofs of remaining results from Section~\ref{sec:small-ball-bounds}, namely Proposition~\ref{prop:small-ball-lower} and Corollary~\ref{cor:negative-moments-smallest}.

\subsection{Proof of Proposition~\ref{prop:small-ball-lower}}
\label{sec:proof-lower-lower}

  Let $\Theta$ be a random variable distributed uniformly on the unit sphere $S^{d-1}$ and independent of $X$.
  We have %
  \begin{align*}
    \sup_{\theta \in S^{d-1}} \P ( | \langle \theta, X\rangle | \leq t )
    &\geq \E %
      \big[ \P ( | \langle \Theta, X\rangle | \leq t \cond \Theta ) \big] %
    = \E \big[ \P %
    ( | \langle \Theta, X\rangle | \leq t \cond X ) \big]
    \, .
  \end{align*}
  Next, note that %
  for every $x \in \R^d$, $\langle \Theta, x\rangle$ is distributed as $\| x \| \cdot \Theta_1$, where $\Theta_1$ denotes the first coordinate of $\Theta$.
  Since $X$ is independent of $\Theta$, the above inequality becomes
  \begin{equation}
    \label{eq:proof_lower_small_ball-1}
    \sup_{\theta \in S^{d-1}} \P ( | \langle \theta, X\rangle | \leq t )
    \geq \E \bigg[ \P \bigg( | \Theta_1 | \leq \frac{t}{\| X \|} \Big| X \bigg) \bigg]
    \, .
  \end{equation}
  Now, since $\E [ \| X \|^2] = \tr ( \E [X X^\top] ) = d$, Markov's inequality implies that $\P ( \| X \| \geq 2 \sqrt{d}) \leq \E [ \| X \|^2 ] / (4d) \leq 1/4$.
  Since $r \mapsto \P_{\theta} ( | \theta_1 | \leq t / r )$ is non-increasing, plugging this into~\eqref{eq:proof_lower_small_ball-1} yields
  \begin{equation}
    \label{eq:proof_lower_small_ball-2}
    \sup_{\theta \in S^{d-1}} \P ( | \langle \theta, X\rangle | \leq t )
    \geq \frac{3}{4} \cdot \P \Big( | \Theta_1 | \leq \frac{t}{2 \sqrt{d}} \Big)
    \, .
  \end{equation}

  Let us now derive the distribution of $|\Theta_1|$.
  Let $\phi: S^{d-1} \to \R$ be the projection on the first coordinate: $\phi (\theta) = \theta_1$ for $\theta \in S^{d-1}$.
  Note that for $u \in [-1, 1]$, $\phi^{-1} (u) = \{ u \} \times ( \sqrt{1 - u^2} \cdot S^{d-2} )$ which is isometric to $\sqrt{1 - u^2} \cdot S^{d-2}$ and hence has $(d-2)$-dimensional Hausdorff measure $C_d (1 - u^2)^{(d-2)/2}$ for some constant $C_d$.
  In addition, since $\phi (\theta) = \langle e_1, \theta\rangle$ (where $e_1 = (1, 0,  \dots, 0)$), $\nabla \phi (\theta) \in (\R \theta)^\perp$ is the orthogonal projection of $e_1$ on $(\R \theta)^\perp$, namely $e_1 - \theta_1 \theta$, with norm $\| \nabla \phi (\theta) \| = \sqrt{1 - \theta_1^2}$.
  Fix $t \in (0,1]$ and define $g (\theta) = \indic{| \theta_1 | \leq t} / \sqrt{1 - \theta_1^2}$, which equals $\indic{|u| \leq t} /\sqrt{1 - u^2}$ on $\phi^{-1} (u)$ (for $u \in (-1, 1)$), and such that $g (\theta) \cdot \| \nabla \phi (\theta) \| = \indic{|\theta_1| \leq t}$.
  Hence, the coarea formula~\cite[Theorem~3.2.2]{federer1996geometric} implies that, for every $t\in (0, 1]$,
  \begin{align}
    \label{eq:density-marginal-uniform-coarea}
    \P ( | \Theta_1 | \leq t )
    &= \int_{S^{d-1}} g (\theta) \| \nabla \phi (\theta) \| \pi (\di \theta)
    = \int_{-1}^1 \frac{\indic{|u| \leq t}}{\sqrt{1 - u^2}} \times C_d (1 - u^2)^{(d-2)/2} \di u \nonumber \\
     &= 2  C_d \int_{0}^t ( 1 - u^2 )^{(d-3)/2} \di u
       \, .
  \end{align}
  If $d=2$,~\eqref{eq:density-marginal-uniform-coarea} implies that $| \Theta_1 |$ has density $(2/\pi) / \sqrt{1 - t^2} \geq 2/\pi$ on $[0, 1]$, and hence %
  for $t \in [0, 1]$:
  \begin{equation}
    \label{eq:proof-lower-smallball-d2}
    \P \Big( | \Theta_1 | \leq \frac{t}{2 \sqrt{d}} \Big)
    \geq \frac{2}{\pi} \times \frac{t}{2 \sqrt{2}}
    \, .
  \end{equation}
  If $d=3$,~\eqref{eq:density-marginal-uniform-coarea} implies that $| \Theta_1 |$ is uniformly distributed on $[0, 1]$, so that for $t \in [0, 1]$ %
  \begin{equation}
    \label{eq:proof-lower-smallball-d3}
    \P \Big( | \Theta_1 | \leq \frac{t}{2 \sqrt{d}} \Big)
    = \frac{t}{2 \sqrt{3}}
    \, .
  \end{equation}
  Now, assume that $d \geq 4$.
  Letting $t=1$ in~\eqref{eq:density-marginal-uniform-coarea} yields the value of the constant $C_d$, which normalizes the right-hand side: since $1-u^2 \leq e^{-u^2}$, 
  \begin{align*}    
    (2 C_d)^{-1}
    &= \int_0^1 (1 - u^2)^{(d-3)/2} \di u 
      \leq \int_0^1 e^{- (d-3) u^2 /2} \di u \\
    &\leq \frac{1}{\sqrt{d-3}} \int_0^{\sqrt{d-3}} e^{-u^2/2} \di u
      \leq \frac{1}{\sqrt{d-3}} \times \sqrt{\frac{\pi}{2}}
      \, ,
  \end{align*}
  so that $2 C_d \geq \sqrt{2(d-3) / \pi}$.
  Finally, if $u \leq 1/(2 \sqrt{d})$, then
  \begin{equation*}
    \left( 1 - u^2 \right)^{(d-3)/2}
    \geq \left( 1 - \frac{1}{4 d} \right)^{d/2}
    \geq \left( 1 - \frac{1}{16} \right)^{2}
    \, ,
  \end{equation*}
  using the fact that $4d \geq 16$ and that the function $x \mapsto (1 - 1/x)^{x/8}$ is increasing on $(1, + \infty)$.
  Plugging the above lower bounds in~\eqref{eq:density-marginal-uniform-coarea} shows that, for $t \leq 1$,
  \begin{equation}
    \label{eq:proof-lower-small-d4}
    \P \Big( | \Theta_1 | \leq \frac{t}{2 \sqrt{d}} \Big)
    = 2  C_d \int_{0}^{t/(2 \sqrt{d})} ( 1 - u^2 )^{(d-3)/2} \di u
    \geq \sqrt{\frac{2 (d-3)}{\pi}} \times \Big( \frac{15}{16} \Big)^2 \frac{t}{2 \sqrt{d}}
    \geq \frac{t}{3}
  \end{equation}
  where the last inequality is obtained by noting that $(d-3) / d \geq 1/4$ for $d \geq 4$ and lower bounding the resulting constant.
  The bounds~\eqref{eq:proof-lower-smallball-d2}, \eqref{eq:proof-lower-smallball-d3} and \eqref{eq:proof-lower-small-d4} imply that, for every $d \geq 2$ and $t \leq 1$,
  \begin{equation}
    \label{eq:proof-lower-smallball-d}
    \P \Big( | \Theta_1 | \leq \frac{t}{2 \sqrt{d}} \Big)
    \geq \frac{t}{\pi \sqrt{2}}
    \, .
  \end{equation}
  The first inequality of Proposition~\ref{prop:small-ball-lower} follows by combining inequalities~\eqref{eq:proof_lower_small_ball-2} and~\eqref{eq:proof-lower-smallball-d}.
  The second inequality~\eqref{eq:small-ball-lower-2} is a consequence of the first by Lemma~\ref{lem:smallest-eigenvalue-min}. %

\subsection{Proof of Corollary~\ref{cor:negative-moments-smallest}}
\label{sec:proof-corollary-moments}

Corollary~\ref{cor:negative-moments-smallest} directly follows from Theorem~\ref{thm:lower-tail-covariance}, Proposition~\ref{prop:small-ball-lower} and Lemma~\ref{lem:negative-moments} below.

\begin{lemma}
  \label{lem:negative-moments}
  Let $Z$ be a nonnegative real variable.
  \begin{enumerate}
  \item If there exist some constants $C \geq 1$ and $a \geq 2$ such that $\P (Z \leq t) \leq (C t)^a$ for all $t > 0$, then $\| Z^{-1} \|_{L^q} \leq \| \max (1, Z^{-1}) \|_{L^q} \leq 2^{1/q} C \leq 2 C$ for all $1 \leq q \leq a/2$.
  \item Conversely, if $\| Z^{-1} \|_{L^q} \leq C$ for some constants $q \geq 1$ and $C > 0$, then $\P (Z \leq t) \leq (C t)^{q}$ for all $t > 0$.
  \item Finally, if there exist constants $c, a > 0$ such that $\P (Z \leq t) \geq (c t )^a$ for all $t \in (0, 1)$, then $\| Z^{-1} \|_{L^q} = + \infty$ for $q \geq a$.
  \end{enumerate}
\end{lemma}

\begin{proof}
  For the first point, since $\max (1, Z^{-q})$ is nonnegative, we have 
  \begin{equation*}
    \E [ \max (1, Z^{-q}) ]
    = \int_0^{\infty} \P (\max (1, Z^{-q}) \geq u) \, \di u
    = \int_0^{\infty} \P (\min (1, Z) \leq u^{-1/q}) \, \di u
    \, .
  \end{equation*}
  For $u \leq C^q$, we bound $\P (\min (1, Z) \leq u^{-1/q}) \leq 1$, while for $u \geq C^q$ (so that $u^{-1/q} \leq C^{-1} \leq 1$), we bound $\P (\min (1, Z) \leq u^{-1/q}) = \P (Z \leq u^{-1/q}) \leq (C u^{-1/q})^a$.
  We then conclude that
  \begin{align*}
    \| \max (1, Z^{-1}) \|_{L^q}^q
    \leq C^q + \int_{C^q}^{\infty} (C^{-q} u)^{- a / q} \di u
    = C^q \bigg[ 1 + \int_1^{\infty} v^{-a/q} \di v \bigg]
    \leq 2 C^q
    \, ,
  \end{align*}
  where we let $v = C^{-q} u$ and used the fact that $\int_1^{\infty} v^{-a/q} \di v \leq \int_1^\infty v^{-2} \di v = 1 $ since $q \leq a/2$.
  The second point follows from Markov's inequality: for every $t > 0$,
  \begin{equation*}
    \P (Z \leq t)
    = \P (Z^{-q} \geq t^{-q})
    \leq t^q \cdot \E [Z^{-q}]
    \leq (C t)^q
    \, .
  \end{equation*}
  Finally, for the third point, since $\P (Z \leq u^{-1/q}) \geq (c u^{-1/q})^a$ for $u > 1$, we have for $q \geq a$:
  \begin{equation*}
    \E [ Z^{-q} ]
    = \int_0^{\infty} \P (Z \leq u^{-1/q}) \di u
    \geq \int_1^{\infty} c^a u^{- a / q} \di u
    \geq c^a \int_1^{\infty} u^{-1} \di u
    = + \infty
    \, .
    \qedhere
  \end{equation*}
\end{proof}

\subsection{Proof of Proposition~\ref{prop:uniform-small-ball}}
\label{sec:proof-proposition-independent}

The proof relies on the following lemma.

\begin{lemma}
  \label{lem:uniform-small-ball-general}
  Let $X^1, \dots, X^d$ be independent real random variables.
  Assume that there exists a sub-additive function $g: \R^+ \to \R$ such that, for every $j = 1, \dots, d$ and $\xi \in \R$,
  \begin{equation*}
    | \Phi_{X^j} (\xi) |
    \leq \exp ( - g (\xi^2))
    \, .
  \end{equation*}
  Then, for every $t \in \R$,
  \begin{equation}
    Q_{X} (t)
    \leq t \cdot \int_{-2\pi/t}^{2\pi/t} \exp (- g (\xi^2)) \, \di \xi
    \, .
  \end{equation}
\end{lemma}

\begin{proof}[Proof of Lemma~\ref{lem:uniform-small-ball-general}]
  For every $\theta \in S^{d-1}$ and $\xi \in \R$, we have, by independence of the $X^j$,
  \begin{align*}
    | \Phi_{\langle \theta, X\rangle} (\xi) |
    &= \big| \E \big[ e^{i \xi (\theta_1 X^1 + \dots + \theta_d X^d)} \big] \big|
      = \big| \E \big[ e^{i \xi \theta_1 X^1} \big] \big| \cdots \big| \E \big[ e^{i \xi \theta_d X^d} \big] \big| \\
    &\leq \exp \big[- \big( g (\theta_1^2 \xi^2) + \dots + g (\theta_d^2 \xi^2) \big) \big]
      \leq \exp ( - g(\xi^2))
      \, ,
  \end{align*}
  where the last inequality uses the sub-additivity of $g$ and the fact that $\theta_1^2 + \dots + \theta_d^2 = \| \theta \|^2 = 1$.
  Lemma~\ref{lem:uniform-small-ball-general} then follows from Esséen's inequality \cite{esseen1966kolmogorov}, which states that for any real random variable~$Z$,
  \begin{equation*}
    Q_Z (t) %
    \leq t \cdot \int_{- 2 \pi/ t}^{2 \pi / t} | \Phi_Z (\xi) | \, \di \xi
    \, .
    \qedhere
  \end{equation*}
\end{proof}

\begin{proof}[Proof of Proposition~\ref{prop:uniform-small-ball}]
  The functions $g_1 : u \mapsto \alpha \log (1 + u)$ and $g_2 : u \mapsto C_0^{-1} \sqrt{u}$ are concave functions on $\R^+$ taking the value $0$ at $0$, and therefore sub-additive.
  Since $g_1$ is also increasing, the function $g : u \mapsto g_1 \circ g_2 (u) = \alpha \log (1 + C_0^{-1} \sqrt{u})$ is also sub-additive.
  Condition~\eqref{eq:bound-characteristic} simply writes $\Phi_{X^j} (\xi) \leq \exp (- g (\xi^2))$, so that by Lemma~\ref{lem:uniform-small-ball-general}
  \begin{align*}    
    Q_X (t)
    \leq t \int_{-2\pi/t}^{2\pi/t} \frac{1}{(1 + |\xi|/C_0)^\alpha} \di \xi
    \leq 2 t \int_{0}^{2\pi/t} \frac{\di \xi}{(\xi/C_0)^\alpha} 
    = \frac{2 t C_0^\alpha}{1 - \alpha} \Big( \frac{2\pi}{t} \Big)^{1 - \alpha}
    \, ,
  \end{align*}
  which implies that $Q_X (t) \leq (C t)^\alpha$, concluding the proof.
\end{proof}

\section{Conclusion}
\label{sec:conclusion}

We analyzed random-design linear prediction from a minimax perspective, by obtaining matching upper and lower bounds on the risk under weak conditions.
This revealed that the hardness of the problem is characterized by the distribution of leverage scores, and that Gaussian design is almost the most favorable one in high dimension.

The upper bounds relied on a study of the lower tail and negative moments of empirical covariance matrices.
We showed a general lower bound on this lower tail in dimension $d \geq 2$, as well as a matching upper bound under a necessary %
regularity condition on the design.
The proof of this result relied on the use of PAC-Bayesian smoothing of empirical processes, with refined non-Gaussian smoothing distributions.

It is worth noting that the upper bound of Theorem~\ref{thm:lower-tail-covariance} on the lower tail of $\lamin (\wh \Sigma_n)$ requires $n \geq 6 d$;
the approach used here is not sufficient to obtain meaningful bounds for (nearly) square matrices, whose aspect ratio $d/n$ is close to $1$.
It could be interesting to see if the bound of Theorem~\ref{thm:lower-tail-covariance} can be extended to this case (for instance in the case of centered, variance $1$ independent coordinates with bounded density, as in Section~\ref{sec:small-ball-indep}%
),
by leveraging the techniques from \cite{rudelson2008littlewood,rudelson2009smallest,tao2009inverse,tao2009littlewood}.

\paragraph{Acknowledgements.}

The author would like to thank two anonymous referees and an associate editor for very helpful comments that improved the quality of this paper.

\paragraph{Funding.}
Part of this work was carried at Centre de Mathématiques Appliquées, École polytechnique, France, and supported by a public grant as part of the Investissement d'avenir project, reference ANR-11-LABX-0056-LMH, LabEx LMH. Part of this work was carried out at the Machine Learning Genoa center, Universit\`a di Genova, Italy.

\bibliographystyle{alpha}
\newcommand{\etalchar}[1]{$^{#1}$}

\end{document}